\documentclass[12pt,a4paper]{article}
\usepackage{amsthm,amsfonts,amsmath,amssymb,graphicx}

\usepackage{graphics}
\usepackage{graphicx,graphics}
\usepackage{amsmath}%
\usepackage{amssymb}%
\usepackage{mathptmx}
\usepackage{color}

\theoremstyle{plain}
\newtheorem{theorem}{Theorem}
\newtheorem{lemma}[theorem]{Lemma}
\newtheorem{corollary}[theorem]{Corollary}
\newtheorem{proposition}[theorem]{Proposition}

\theoremstyle{definition}
\newtheorem{definition}[theorem]{Definition}
\newtheorem{remark}[theorem]{Remark}

\newtheorem{comment}{Commment}
\theoremstyle{remark}

\DeclareMathOperator{\Lip}{Lip}
\DeclareMathOperator{\dom}{dom}

\def\bq{\begin{eqnarray}}
\def\eq{\end{eqnarray}}
\def\bqq{\begin{eqnarray*}}
\def\eqq{\end{eqnarray*}}
\def\nn{\nonumber}

\def\eps{\varepsilon}

\def\N{\mathbb{N}}
\def\R{\mathbb{R}}

\def\cF {\mathcal{F}}

\def\cR {\mathbb{R}}

\def\cP {\mathcal{P}}

\title{Some results of the Lipschitz constant of 1-Field on $\R^n$}
\author{Erwan Y. Le Gruyer \thanks{   E. Y. Le Gruyer,  INSA de Rennes \& IRMAR, 
20, Avenue des Buttes de Co\"{e}smes, 
CS 70839 F - 35708 Rennes Cedex 7 , France,
(Erwan.Le-Gruyer@insa-rennes.fr)} 
~and~  Thanh-Viet Phan \thanks{T.V. Phan,  INSA de Rennes \& IRMAR, 
20, Avenue des Buttes de Co\"{e}smes, 
CS 70839 F - 35708 Rennes Cedex 7 , France,
(Thanh-Viet.Phan@insa-rennes.fr)     }}

\begin{document}
\maketitle
\begin{abstract}
We study the relations between the Lipschitz constant of $1$-field introduced in \cite{ELG1} and the Lipschitz constant of the gradient canonically associated with this $1$-field. Moreover, we produce two explicite formulas that make up Minimal Lipschitz extensions for $1$-field. 
As consequence of the previous results, for the problem of minimal extension by continuous functions from $\R^m$ to $\R^n$, we also produce analogous explicite formulas to those of  Bauschke and Wang (see \cite{Ba}).
 Finally, we show that Wells's extensions of $1$-field are absolutely minimal Lipschitz extension  when the domain of $1$-field to expand is finite. We provide a counter-example showing that this result is false in general.
\end{abstract}

\section{Introduction}

Fix $n \in \N$, $n\ge 1$. Let $\Omega$ be a non-empty subset of Euclidean space $\cR^n$. We denote $\left\langle , \right\rangle$ the standard scalar product in $\R^n$. Let $\cP^1(\mathbb{R}^n,\mathbb{R})$ be the set of first degree polynomials mapping $\mathbb{R}^n$ to $\mathbb{R}$, i.e 
$$
\cP^1(\R^n,\R)\triangleq\{P:a\in\R^n\mapsto P(a)=p+\left\langle v,a\right\rangle,where \text{ }p\in\R,v\in \R^n\}.
$$
Let us consider a $1$-field $F$  on domain $\dom(F)  \triangleq \Omega$ defined by
\begin{align}\label{def-field}
F: \, &\Omega \rightarrow \cP^1(\R^n,\R) \nn\\
&x \mapsto F(x)(a) \triangleq f_x + \langle D_xf ; a - x\rangle,
\end{align}
where $a \in \R^n$ is the evaluation variable of the polynomial $F(x)$ and $x\in\Omega \mapsto f_x \in \R$, $x\in\Omega \mapsto D_xf \in \R^n$ are mappings associated with $F$. We will always use capital letters to denote the $1$-field and small letters to denote these mappings.

The Lipschitz constant of $F$ introduced in \cite{ELG1} is 
\bq\label{dncuaGamma}
 \Gamma^1(F;\Omega) \triangleq \sup_{\substack{x,y \in \Omega \\ x
    \neq y}} \Gamma^1(F;x,y),
\eq
where
\bq
 \Gamma^1(F;x,y) \triangleq 2 \sup_{a \in \R^n} \frac{|F(x)(a) -
  F(y)(a)|}{\|x-a\|^2 + \|y-a\|^2}. 
\eq

If $\Gamma^1(F;\Omega)< +\infty$, then the Whitney's conditions \cite{Whitney2}, \cite{Glaeser1} are verified and the 1-field $F$ can be extended in $\R^n$: there exists $g\in \mathcal{C}^1(\cR^n,\cR)$ such that $g(x)=f_x$ and $\nabla g(x)=D_xf$ for all $x\in \Omega$ where $\nabla g$ is the usual gradient.
Moreover, from \cite[Theorem 2.6]{ELG1} we can find $g$ which satisfies $$\Gamma^1(G;\R^n)=\Gamma^1(F;\Omega),$$
where $G$ is the 1-field associated to $g$, i.e
$$G(x)(y)=g(x)+\left\langle \nabla  g(x),y-x \right\rangle,~ \forall x,y\in \Omega.$$
It means that the Lipschitz constant does not increase when extending $F$ by $G$. We say that $G$ is a minimal Lipschitz extension (MLE for short) of $F$ and we have
\begin{align*}
 \Gamma^1(G;\R^n)= {\rm inf} \{\Lip(\nabla h,\R^n)~:~  h(x) = f_x,\nabla h(x)=D_xf, x\in\Omega,~ h \in \mathcal{C}^{1,1}(\R^n,\R) \},
\end{align*}
where the notation $\Lip(u,.)$ means that 

\bq\label{dncualip}
 {\rm Lip}(u;x,y)  \triangleq   \frac{\|u(x)-u(y)\|}{\|x-y\|},~x \neq y \in \Omega,   \mbox{ and }    {\rm Lip}(u;\Omega)\triangleq\sup\limits_{x\ne y\in\Omega} {\rm Lip}(u;x,y).
\eq


It is worth asking what is it the relationship between $\Gamma^1(F,\Omega)$ and $  {\rm Lip}(Df,\Omega)$ ? From \cite{ELG1}, we know that $  {\rm Lip}(Df,\Omega)\le \Gamma^1(F,\Omega)$. In special case $\Omega=\cR^n$ we have $  {\rm Lip}(Df,\cR^n)= \Gamma^1(F,\cR^n)$ but in general the formula $ {\rm Lip}(Df,\Omega)= \Gamma^1(F,\Omega)$ is untrue.   
In this paper we will prove that if 
$\Omega$ is an non empty open subset of $\R^n$ then
\begin{equation}
  \Gamma^1(F,\Omega) = {\rm max}\{\Gamma^1(F, \partial\Omega),  {\rm Lip}(Df,\Omega) \},
\end{equation}
where  $\partial \Omega$ is a boundary of $\Omega$.\\
Moreover if $\Omega$ is a convex subset of $\R^n$ then
\begin{equation}
  \Gamma^1(F,\Omega) \leq 2    {\rm Lip}(Df,\Omega).
\end{equation}
Knowing the set of uniqueness of minimal extensions of a taylonian  biponctual field, allows  better understand the relation between $\Gamma^1(F)$ and ${\rm Lip}(Df)$.
For further more details see Section \ref{chap 2}.

In Section \ref{chap 3}, we present two MLEs $U^+$ and $U^-$  of F  of the form
$$
U^+: \,  x \in \R^n  \mapsto U^+(x)(a) \triangleq u^+(x) + \langle D_xu^+ ; a - x\rangle, a \in \R^n.
$$
where

$$
u^+(x)  \triangleq \max\limits_{v\in\Lambda_x}\inf\limits_{a\in\Omega} \Psi^+(F,x,a,v),~~~ D_xu^+\triangleq \mbox{arg}\displaystyle\max_{v\in\Lambda_x}\inf\limits_{a\in\Omega} \Psi^+(F,x,a,v),
$$
and

$$
U^-: \,  x \in \R^n  \mapsto U^-(x)(a) \triangleq u^-(x) + \langle D_xu^- ; a - x\rangle,~ a \in \R^n,
$$
where

$$
u^-(x)  \triangleq \min\limits_{v\in\Lambda_x}\sup \limits_{a\in\Omega} \Psi^-(F,x,a,v),~~~ D_xu^-\triangleq \mbox{arg}\displaystyle\min_{v\in\Lambda_x}\sup\limits_{a\in\Omega} \Psi^-(F,x,a,v),
$$

where $\Lambda_x$ is a non empty and  convex set of $\R^n$, defined in Definition \ref{Def.14}.  
These maps and their gradients are explicit $\sup$-$\inf$ formulas that is to say they only depend on $F$.  In addition they are {\it extremal}  : the first is \textit{ over} and the second is \textit{ under }that is to say 

$$u^-(x)  \le g_x  \le  u^+(x),~ \forall x\in\cR^n,$$
for all MLE $G$ of $F$.

Now, we draw the connection between  the formula $u^+$ and the formula of Wells \cite{Well1}. From \cite[Theorem 2]{Well1}, we know that if $\kappa>0$ satisfies 
\bqq
f_y\le f_x+\frac{1}{2}\langle D_xf+D_yf,y-x\rangle +\frac{\kappa}{4}(x-y)^2-\frac{1}{4\kappa}(D_xf-D_yf)^2,\forall x,y\in\Omega,
\eqq
then there exists $w^+\in \mathcal{C}^{1,1}(\R^n,\R)$ such that $w^+(x)=f_x$, $\nabla w^+(x)=D_xf$ for all $x\in \Omega$, and $\Lip(\nabla w^+,\cR^n)\le \kappa$.

Further, if $g\in \mathcal{C}^{1,1}(\R^n,\R)$ 
with $g(x)=f_x$, $\nabla g(x)=D_xf$ for all $x\in \Omega$ and  $\Lip(\nabla g,\cR^n)\le \kappa$, then $g(x)\le w^+(x)$ for all $x\in \cR^n$.

The construction of Wells  $w^+$ is explicit when $\Omega$ is finite. The construction of Wells  extend to infinite $\Omega$ by passing to the limit but there is no explicit formula. In sections \ref{chap 3} and \ref{chap 6}, we will prove that if $\kappa$ is assigned the Lipschitz constant of the field $F$,  then $w^+$ is a MLE. Further, in this case $w^+$  is an over extremal extension of $F$ and $u^+=w^+$.  We find a function $w^-$ (similarly from the construction of $w^+$) and we have $u^-=w^-$. 

We pay attention  to the case when $\Omega$ is finite. Say it is interesting despite this restrictive assumption. In this case we have  explicit constructions of $w^\pm$. In section \ref{chap 5}, we will prove that $w^{\pm}$ are  absolutely minimal Lipschitz extensions (AMLEs for short) of $F$.
This means that for any bounded open $D$ satisfying $\overline{D}\subset  \cR^n\backslash \Omega$ we have 
 \begin{equation}
  \Gamma^1(w^{\pm},D) = \Gamma^1(w^{\pm},\partial D).
\end{equation}
These result give the existence of AMLEs of $F$ when $\Omega$ is finite. In general one dose not have uniqueness since it may happen $w^-< w^+$. In fact, we even have infinity solutions AMLE of $F$ (see Corollary \ref{Coro.infinity} ).

When $\Omega$ is infinite, $w^+$ and $w^-$ are extremal MLE, but in general are not AMLE of $F$. To prove this, we present, in section \ref{chap 5}, an example of mapping $F$ for which $w^+$ and $w^-$ are not AMLE of $F$. In this particular example, we can check that $\dfrac{1}{2}(w^++w^-)$ is the unique  AMLE of $F$, moreover this function is not  $\mathcal{C}^2$ although the domain $\Omega$ of $F$ is  regular and $F$  is a regular 1-field. 
The question of the existence of an AMLE remains an open and a difficult problem when $\Omega$ is  infinite  see \cite{ELG2} and the references therein.
 

In Section \ref{chap 4}, we explain how to use the previous ideas and methods to construct MLE of mappings from $\cR^m$ to $\cR^n$, i.e,  to solve the Kirszbraun-Valentine extension problem \cite{Kirs,Va}.  Let us define first $\mathcal{Q}_0$  as the problem of the minimum extension for Lipschitzian functions,  second  $\mathcal{Q}_1$  as the problem of the minimum extension for $1$-fields. Curiously, we will show that the problem  $\mathcal{Q}_0$ is a sub-problem of the problem   $\mathcal{Q}_1$. As a consequence, we obtain two explicit formulas see (\ref{Over-Kir2}) and (\ref{Under-Kir2}) that solve the problem  $\mathcal{Q}_0$. \\
The Bauschke-Wang result \cite{Ba} gives an explicite formula for the Kirsbraun-Valentine problem from $\R^m$ to $\R^n$. By our approach, we produce analogous formulas. Moreover, when the domain of the function  to extend is  finite, the result of Wells gives an explicite construction of minimal Lipschitz extensions.




\section{Preliminaries}
For any  $\Omega$ open subset of   $\cR^n$, denote by  $\mathcal{C}^{1,1}(\Omega,\cR)$ the set of all real-valued function $f$ that is differentiable on $\Omega$ and the differential  $\nabla  f$  is lipschitzian, that is ${\rm Lip}(\nabla f,\Omega) < +\infty$.\\

Let $\Omega$ be a subset of $\cR^n$. The $1$-field $F$ of domain $\Omega$  is define by  (see(\ref{def-field}))

\begin{align}
F: \, &\Omega \rightarrow \cP^1(\R^n,\R) \nn\\
&x \mapsto F(x)(a) = f_x + \langle D_xf ; a - x\rangle, a \in \R^n
\end{align}
with $f_x \in \R$ and $D_xf \in \R^n$, fror all $x \in \R^n$.

\begin{definition}
We call $F$ to be a Taylorian field on $\Omega$ if $F$ is a 1-field on $\Omega$ and $\Gamma^1(F,\Omega)<+\infty$. 
\end{definition}
Denote by  $\cF^1(\Omega)$  the set of all Taylorian fields on $\Omega$.

\textbf{Information and precision for the reader} : Let $\Omega$ be a subset of $\cR^n$ and  $F \in \cF^1(\Omega)$. Let us define the map $$f(x) \triangleq F(x)(x)=f_x,~ x \in \Omega.$$
Suppose that $\Omega$ is open, using \cite[Proposition 2.5]{ELG1} we have  $f \in 
\mathcal{C}^{1,1}(\Omega,\cR)$ and $\nabla f(x) =  D_xf$, $\forall x \in \Omega$.\\
Suppose that $\Omega$ is any subset of $\R^n$ then using  \cite[Theorem 2.6]{ELG1} there exists $\tilde{F} \in \cF^1(\R^n)$ which extends $F$. Moreover $\tilde{f} \in  \mathcal{C}^{1,1}(\R^n,\cR)$ and  $ \nabla f(x)  \triangleq  \nabla  \tilde{f}(x) =  D_xf$, $x \in \Omega$.
Conclusion : in all situations, we can  canonically associate  $ F $ and $ f $.  


Using the notation $V\subset\subset \Omega$ if $\overline{V}$ is compact and $\overline{V}\subset\Omega$.\\
Let $x,y\in\cR^n$. We define:
\bqq
B(x;r)\triangleq\{y\in \cR^n:\|y-x\|< r\}.
\eqq
\bqq
\overline{B(x;r)}\triangleq\{y\in \cR^n:\|y-x\|\le r\}.
\eqq
$B_{1/2}(x,y)$ is the closed ball of center $\frac{x+y}{2}$ and radius $\frac{\|x-y\|}{2}$.\\
The line segment joining two points $x$ and $y$ is denote by $[x,y]$, i.e $[x,y]\triangleq\{tx+(1-t)y:0\le t\le 1\}$. \\
The $|$ symbol designates in restriction to.

\begin{definition} Let $\Omega$ be a subset of $\mathbb{R}^n$ and let $F\in \cF^1(\Omega)$. For any $a\neq b \in \Omega$, we define
\bqq
A_{a,b}(F)&\triangleq& \frac{2(f_a-f_b)+\left\langle{D_af+D_bf,b-a}\right\rangle}{\left\| {a-b} \right\|^2}. \hfill\\
B_{a,b}(F) &\triangleq& \frac{\left\| {D_af-D_bf} \right\|}{\left\| {a-b} \right\|} .\hfill\\
\eqq
\end{definition}

\begin{definition}\label{def:quantrongnhat}
Let $\Omega_1\subset \Omega_2\subset \cR^n$ and $F\in\cF^1(\Omega_1)$.

We call $G\in\cF^1(\Omega_2)$ a {\it extension} of $F$ on $\Omega_2$ if $G(x)=F(x)$ for $x\in\Omega_1$.

We say that $G\in\cF^1(\Omega_2)$ is a {\it minimal Lipschitz extension} (MLE) of $F$ on $\Omega_2$ if $G$ is an extension of $F$ on $\Omega_2$ and $$\Gamma^1(G;\Omega_2)=\Gamma^1(F;\Omega_1).$$

We say that  $G_1\in\cF^1(\Omega_2)$ is an {\it over extremal Lipschitz extension} (over extremal for short) and   $G_2$ is an   {\it under extremal Lipschitz extension} of $F$   on $\Omega_2$  if $G_1$ and $G_2$ are MLEs of $F$ on $\Omega_2$ and 
$$
g_2(x) \leq k(x) \leq g_1(x),~ x \in \Omega_2,
$$
for all $K$  MLE of $F$.

We say that $G\in\cF^1(\Omega_2)$ is an {\it absolutely minimal Lipschitz extension} (AMLE) of $F$ on $\Omega_2$ if $G$ is a MLE of $F$ on $\Omega_2$ and
 $$\Gamma^1(G;V)=\Gamma^1(F;\partial V), $$
for any bounded open V satisfying $\overline{V}\subset \Omega_2\backslash \Omega_1$.

\end{definition}

We recall some results in \cite{ELG1} that will be useful in sections \ref{chap 2} and \ref{chap 3}:
\begin{proposition} \cite[Proposition 2.2 and remark 2.3]{ELG1}  \label{pro:halAB} Let $\Omega$ be a subset of $\mathbb{R}^n$ and let $F\in \cF^1(\Omega)$ then for any $a, b\in \Omega, a \ne b$ we have
\bqq
\Gamma^1(F;a,b)&=&\sqrt{A_{a,b}(F)^2+B_{a,b}(F)^2}+\left| A_{a,b}(F)\right|=2\sup\limits_{y\in B_{1/2}(a,b)}\frac{F(a)(y)-F(b)(y)}{\|a-y\|^2+\|b-y\|^2}.\\
\eqq
\end{proposition}

\begin{theorem}\cite[Theorem 2.6]{ELG1} \label{theo:moronggafa} Let $\Omega_1\subset\Omega_2\subset \R^n$ and let $F\in \cF^1(\Omega_1)$ then there exists a MLE $G\in \cF^1(\Omega_2)$ of $F$ on $\Omega_2$.  
\end{theorem}





\section{Relations between $\Gamma^1(F;\Omega)$ and $\Lip(Df;\Omega)$}\label{link gamma and lip}\label{chap 2}
In this section $\Omega$ is an open subset of $\R^n$. Let $F\in\cF^1(\Omega)$. 

From Proposition \ref{pro:halAB}, we have  $\Gamma^1(F;\Omega)  \geq  \Lip(Df;\Omega)$. When $\Omega =  \R^n$, we know that (see \cite[Proposition 2.4]{ELG1})
 $$\Lip(Df,\cR^n)= \Gamma^1(F,\cR^n),$$ 
but in general $\Gamma^1(F;\Omega)$ may be strictly bigger than $\Lip(Df;\Omega)$. For example, let $A$ and $B$ be open sets in $\cR^n$ such that $A\cap B=\emptyset$. Let $\Omega=A\cup B$, then $\Omega$ is open. Let  $F\in \cF^1(\Omega)$ such that $f_x=0$ if $x\in A$, $f_x=1$ if $x\in B$, and $D_xf=0$ $,\forall x\in \Omega$. Then  $\Lip(Df;\Omega)=0$ and from Proposition  \ref{pro:halAB} we have $$\Gamma^1(F;\Omega)=\sup_{x\in A}\sup_{y \in B} \frac{4}{\|x-y\|^2}>0.$$ 

We now give two new results where we have $\Gamma^1(F,\Omega)=\Lip(Df,\Omega)$.
\begin{proposition}\label{pro:lipagatp} Let $F\in\cF^1(\Omega)$. Suppose there exist $a,b\in \Omega,a\ne b$ such that 
$\Gamma ^1(F;a,b)=\Gamma ^1(F;\Omega)$, then $\Gamma^1(F;\Omega)=\Lip(Df;\Omega)$. 
\end{proposition}
\begin{proof}
It is enough to prove that $\Gamma^1(F;\Omega)\le \Lip(Df;\Omega)$.

Let $G=F|_{\{a,b\}}$ be a Taylorian field on $\dom(G)=\{a,b\}$ with $G(a)=F(a)$ and $G(b)=F(b)$. Let $U$ be a MLE of $F$ on $\R^n$. We have $U(a)=F(a)=G(a)$, $U(b)=F(b)=G(b)$ and $\Gamma^1(U;\mathbb{R}^n)=\Gamma^1(F;\Omega)=\Gamma^1(F;a,b)=\Gamma^1(G;{\rm dom}(G))$. Therefore $U$ is a MLE of $G$ on $\R^n$. 

Using \cite[Lemma 8 and Lemma 10]{ELG2}, there exists a point $c\in B_{1/2}(a,b)$ such that 
\bq\label{ctqurnan}
{\rm Lip}(Du;x,y)  = {\rm Lip}(Du;s,t)  =\Gamma ^1(G;a,b) =\Gamma ^1(F;a,b),
\eq

for all $x, y\in[a,c]$ ($x\ne y$) and $s,t\in[b,c]$ ($s\ne t$).

Since $a\ne b$, we have $c\ne a$ or $c\ne b$. We can assume $c\ne a$. Because $\Omega$ is open, there exists $x\ne y\in [a,c]\cap\Omega$, thus from (\ref{ctqurnan}) we have
\bqq
\Lip(Df;\Omega)\ge {\rm Lip}(Df;x,y)  = {\rm Lip}(Du;x,y)  = \Gamma^1(F;a,b)=\Gamma^1(F;\Omega).
\eqq
\end{proof}

\begin{proposition}\label{pro:lhomep}
 Let $F\in\cF^1(\Omega)$. Suppose there exists $\Omega'\subset\subset\Omega$ such that $\Gamma^1(F;\Omega')=\Gamma^1(F;\Omega)$, then $\Gamma^1(F;\Omega)=\Lip(Df;\Omega)$.
\end{proposition}
\begin{proof}Again, it is enough to prove that $\Gamma^1(F;\Omega)\le \Lip(Df;\Omega)$.
Let $h>0$, we define 
$\Lambda_h=\{(a,b)\in\overline{\Omega'}\times\overline{\Omega'}:|a-b|\ge h\}$ and $\Gamma^1_h(F;\overline{\Omega'})=\sup\limits_{(a,b)\in\Lambda_h}\Gamma^1(F;{a,b})$. Applying Proposition \ref{pro:halAB}, the mapping $(a,b)\mapsto\Gamma^1(F;a,b)$ is continuous  on $\Lambda_h$. Moreover, $\Lambda_h$ is compact, thus there exists $(a_h,b_h)\in\Lambda_h$ such that 
\bq\label{eq:fahbhj}
\Gamma^1(F;a_h,b_h)=\Gamma^1_h(F;\overline{\Omega'}).
\eq
{\it*Case 1:} There exists $h>0$ such that 
\bq\label{eq:fahbhjc1}
\Gamma^1_h(F;\overline{\Omega'})=\Gamma^1(F;\overline{\Omega'})
\eq
From (\ref{eq:fahbhj}),(\ref{eq:fahbhjc1})  and the condition $\Gamma^1(F;\Omega')=\Gamma^1(F;\Omega)$, we have $\Gamma^1(F;a_h,b_h)=\Gamma^1(F;\Omega)$. Applying Proposition \ref{pro:lipagatp} we have $\Gamma^1(F;\Omega)=\Lip(Df;\Omega)$.\\
{\it*Case 2:} For all $h>0$, we always have 
\bq\label{eq:fahbhjc2}
\Gamma^1_h(F;\overline{\Omega'})<\Gamma^1(F;\overline{\Omega'}).
\eq

Let $h=1/n$, then for any $n\in \mathbb{N}$ there exist $(a_n,b_n)\in \Lambda_{1/n}$ such that $\Gamma^1(F;a_n,b_n)=\Gamma^1_{1/n}(F;\overline{\Omega'})$. Since $(a_n),(b_n)\subset \overline{\Omega'}$ and $\overline{\Omega'}$ is compact, there exist a subsequence $(a_{n_k})$ of $(a_n)$ and a subsequence $(b_{n_k})$ of $(b_n)$ such that $(a_{n_k})$ converges to an element $a$ of $\overline{\Omega'}$ and  $(b_{n_k})$ converges to an element $b$ of $\overline{\Omega'}$.

If $a\ne b$ then 
\bqq
\Gamma^1(F;\overline{\Omega'})=\lim\limits_{k\to\infty}\Gamma^1_{1/n_k}(F;\overline{\Omega'})=\lim\limits_{k\to\infty}\Gamma^1(F;a_{n_k},b_{n_k})=\Gamma^1(F;a,b).
\eqq
But this is not possible because for $l=|a-b|>0$ we deduce from (\ref{eq:fahbhjc2}) that $\Gamma^1(F;a,b)\le \Gamma^1_l(F;\overline{\Omega'})<\Gamma^1(F;\overline{\Omega'})$.

Therefore, we must have $a=b$. From the proof of \cite[Proposition 2.4]{ELG1}, we see that if $B_{1/2}(x,y)\subset\Omega$ then $\Gamma^1(F;x,y)\le \Lip(Df;\Omega)$. We will use this property for proving in the case $a=b$. For any $\eps>0$, since $(a_{n_k})$ and $(b_{n_k})$ are both converge to $a\in \overline{\Omega'}\subset\Omega$, there exists $k\in\mathbb{N}$ such that $B_{1/2}(a_{n_k},b_{n_k})\subset\Omega$ and $$\eps+\Gamma^1(F;a_{n_k},b_{n_k})\ge\Gamma^1(F;\overline{\Omega'})=\Gamma^1(F;\Omega).$$ Since $B_{1/2}(a_{n_k},b_{n_k})\subset\Omega$ we have  $$\Gamma^1(F;a_{n_k},b_{n_k})\le \Lip(Df;\Omega).$$Therefore $$\eps+\Lip(Df;\Omega)\ge \Gamma^1(F;\Omega).$$ This inequality holds for any $\eps>0$, so that we have $\Lip(Df;\Omega)\ge \Gamma^1(F;\Omega)$.
\end{proof}

\begin{proposition}\label{pro:be2}
Let $\Omega$ be an open and convex set in $\R^n$ and let $F\in\cF^1(\Omega)$. Then 
\bqq
\Gamma^1(F;\Omega)\le 2\Lip(Df;\Omega).
\eqq
\end{proposition}
\begin{proof}
Let $f$ be the canonical associate to $F$. We can write
\bqq
f(x)-f(y)&=&\int\limits_0^1 \langle \nabla f(y+t(x-y)),x-y\rangle dt\\
\eqq

For any $x,y\in\Omega$ and $z\in\R^n$ we have
\bqq
&&F(x)(z)-F(y)(z)\\
&&=f(x)-f(y)+\langle \nabla f(x),z-x\rangle-\langle \nabla f(y),z-y\rangle\\
&&=\int\limits_0^1 \langle \nabla f(y+t(x-y))-\nabla f(x),x-z\rangle dt+\int\limits_0^1 \langle \nabla f(y+t(x-y))-\nabla f(y),z-y\rangle dt.
\eqq
Hence
\bqq
&&\left|F(x)(z)-F(y)(z)\right|\\
&&\le\int_0^1 \Lip(\nabla f;\Omega)\|x-y\|\|x-z\|(1-t) dt+\int_0^1 \Lip(\nabla f;\Omega)\|(x-y)\|\|z-y\|t dt,\\
&&=\frac{1}{2}\Lip(\nabla f;\Omega)\|x-y\|(\|x-z\|+\|z-y\|),\\
&& \leq \frac{1}{2} \Lip(\nabla f;\Omega) (\|x-z\|+\|z-y\|)^2,\\
&&\le \Lip(\nabla f;\Omega)\left(\|x-z\|^2+\|z-y\|^2\right).
\eqq
Therefore $\Gamma^1(F;\Omega)\le 2\Lip(\nabla f;\Omega)=2\Lip(Df;\Omega)$.
\end{proof}
From the proof of Proposition \ref{pro:be2}, we obtain
\begin{corollary} \label{cor:gam<lip}
Let $\Omega$ be an open and convex set in $\R^n$ and $f\in \mathcal{C}^{1,1}(\Omega,\mathbb{R})$.  Then $F\in\cF^1(\Omega)$ where $F$ is the 1-field associated to $f$.
\end{corollary}

\begin{lemma}\label{lem:libe}
Let $u\in \mathcal{C}^{1}(\R^n,\R)$ then 

\bqq
{\rm Lip}( \nabla  u;x,y) \leq \inf_{z \in [x,y]}  {\rm max}\{{\rm Lip}( \nabla  u;x,z) , {\rm Lip}( \nabla  u;z,y) \}, ~ \mbox{ for all }  x,y\in \R^n.
\eqq
\end{lemma}
\begin{proof}
Let $x,y\in\Omega$ and $z\in [x,y]$ then we have $\|x-y\|=\|x-z\|+\|z-y\|$. It follows 

\bqq
\begin{array}{lll}
{\rm Lip}( \nabla  u;x,y)  & \le & \dfrac{\|x-z\|}{\|x-z\|+\|y-z\|}{\rm Lip}( \nabla  u;x,z)+\dfrac{\|z-y\|}{\|x-z\|+\|y-z\|}{\rm Lip}( \nabla  u;z,y), \\
                           &      & \\
                           & \le & \max\left\{  {\rm Lip}( \nabla  u;x,z) , {\rm Lip}( \nabla  u;z,y)      \right\}.
\end{array}
\eqq

\end{proof}

\begin{proposition}\label{co-ex:lipgam}
There exist $\Omega$ open convex set and $F\in\cF^1(\Omega)$ such that 
\bqq
\Lip(\nabla f;\Omega)<\Gamma^1(F;\Omega).
\eqq
\end{proposition}
\begin{proof}
We consider in $\cR^2$. Let $a=(-1,0), b=(1,0)$. We define $U\in\cF^1(\{a,b\})$ as 
$$  D_au=(0,1),~D_bu=(0,-1),~ u_a=\frac{1}{\sqrt{3}},~ u_b=-\frac{1}{\sqrt{3}}. $$

We have   $A_{a,b}(U)=\frac{1}{\sqrt{3}},~B_{a,b}(U)=1, $ and  $\Gamma^1(F;\{a,b\})=\sqrt{3}.$ 

Using \cite[Lemma 7,8]{ELG2}, we define 
\bqq
&&\kappa=\Gamma^1(F,\{a,b\})=\sqrt{3},\\
&&c=\frac{a+b}{2}+\frac{D_a u-D_b u}{2\kappa},\\
&&\widetilde u_c=U(a)(c)-\frac{\kappa}{2}\|a-c\|^2= U(b)(c) + \frac{\kappa}{2}\|b-c\|^2,\\
&&D_c \widetilde u=D_au+\kappa(a-c)  = D_bu-\kappa(b-c),\\
&& \widetilde U(c)(z)= \widetilde u_c+\langle D_c \widetilde u,z-c\rangle,~ z\in\cR^n,
\eqq
and define 
\bqq
f(z) \triangleq \left\{ \begin{array}{ll}
\widetilde U(c)(z)- \dfrac{\kappa}{2}\dfrac{\langle z-c,a-c\rangle^2}{\|a-c\|^2}, &\text{if }p(z)\ge 0\text{ and }q(z)\le 0, \\
\widetilde U(c)(z)+ \dfrac{\kappa}{2}\dfrac{\langle z-c,b-c\rangle^2}{\|b-c\|^2}, &\text{if }p(z)\le 0\text{ and }q(z)\ge 0,  \\
\widetilde U(c)(z),&\text{if }p(z)\le 0\text{ and }q(z)\le 0,  \\ 
\widetilde U(c)(z)- \dfrac{\kappa}{2}\dfrac{\langle z-c,a-c\rangle^2}{\|a-c\|^2}+ \dfrac{\kappa}{2}\dfrac{\langle z-c,b-c\rangle^2}{\|b-c\|^2},& \text{if }p(z)\ge 0\text{ and }q(z)\ge 0,  \\ 
\end{array}  \right.
\eqq
where $p(z)=\langle a-c,z-c\rangle $  and $q(z)=\langle b-c,z-c\rangle$. 

From \cite[Lemma 8]{ELG2}, we know that if $F$ is the 1-field associated to $f$, then $F$ is a MLE of $U$ on $\cR^2$. Moreover, from \cite[Lemma 9]{ELG2} we have 
\bq \label{eq:Axy=0}
A_{x,y}(F)=0,
\eq
if $$x,y\in \{z\in\cR^2:p(z)\ge 0\text{ and }q(z)\le 0\},$$ or $$x,y\in \{z\in\cR^2:p(z)\le 0\text{ and }q(z)\ge 0\},$$ or $$x,y\in \{z\in\cR^2:p(z)\le 0\text{ and }q(z)\le 0\},$$ or $$x,y\in \{z\in\cR^2:p(z)\ge 0\text{ and }q(z)\ge 0\}.$$

We define \[\left\{ \begin{gathered}
  {w_a=h-\frac{1}{\beta^2}v} \hfill \\
  {w_b=-h-\frac{1}{\beta^2}v} \hfill \\ 
\end{gathered}  \right.\]
where $h=\frac{b-a}{2},~v=\frac{D_au-D_bu}{2\kappa}$ and $\beta=\|v\|\in (0,1)$.

Since $\|h\|=1,~\beta=\|v\|$, $\langle h,v \rangle=0$, $c-a=h+v$ and $c-b=-h+v$, we have 
$$\langle c-a,w_a \rangle=0 {\rm ~and~} \langle c-b,w_b \rangle=0.$$ 

There exists $\alpha_0\in (0,+\infty)$ and  $\alpha_0$ is small such that $x_a=c+\alpha_0w_b,$ and  $x_b=c+\alpha_0w_a$ are in the interior of convex hull of $\{a,b,c\}$. We define 
\bqq
\Delta_a&=&\{x\in\cR^2: x=c+\alpha w_b, \alpha\ge \alpha_0\},\\
\Delta_b&=&\{x\in\cR^2: x=c+\alpha w_a, \alpha\ge \alpha_0\},
\eqq
and
$$
\omega_1 =\Delta_a \cup \{a\},~ \omega_2 =\Delta_a \cup \Delta_b,~ \omega_3 =\Delta_b \cup \{b\}.
$$
\begin{figure}[ht]
\begin{center}
\includegraphics[scale=0.4]{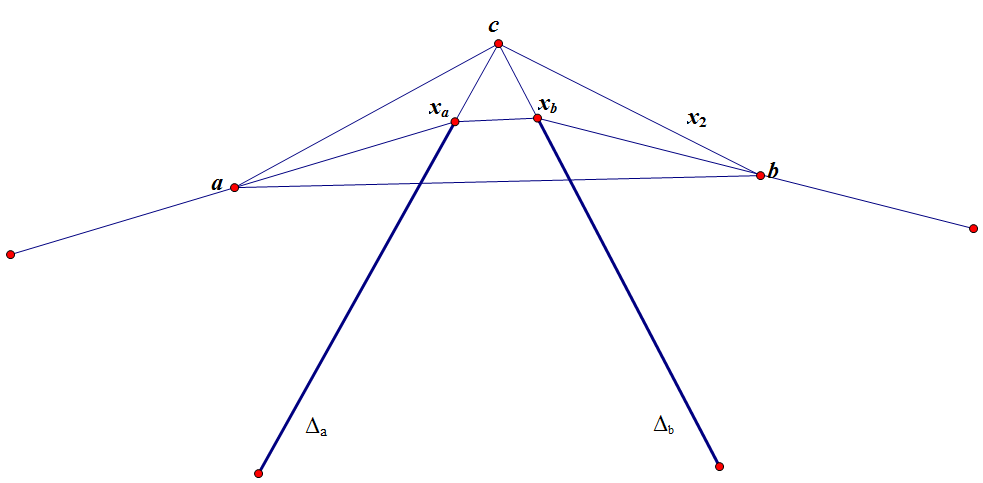}\\
\end{center}
\caption{}
\end{figure}
*Step 1: In this step, we will prove that there exist a constant $k\in(0,1)$ depending on $\beta$ and $\alpha_0$ such that
\bqq
\max_{i =1,2,3} \Gamma^1(F;\omega_i)  \le k\kappa.
\eqq

Since $\|h\|=1,~\|v\|=\beta\in (0,1) $, $c-a=h+v,~ c-b=-h+v,~\langle h,v \rangle=0$,~ $\langle c-a,w_a \rangle=0$ and $\langle c-b,w_b \rangle=0,$  we have\\
*If $x,x'\in\Delta_a$, then

\bq\label{eq:pvdlpg1}
\dfrac{1}{\kappa}{\rm Lip}( \nabla  f;x,x')&=&\frac{\|\langle x-x',a-c\rangle\|}{\|x-x'\|\|a-c\|}=1-\frac{(\beta-1)^2}{1+\beta^2}.
\eq
*If $x,x'\in\Delta_b$, then 
\bq\label{eq:pvdlpg2}
\dfrac{1}{\kappa}{\rm Lip}( \nabla  f;x,x')&=&\frac{\|\langle x-x',b-c\rangle\|}{\|x-x'\|\|b-c\|}=1-\frac{(\beta-1)^2}{1+\beta^2}.
\eq
*If $x\in\Delta_a$, $x'\in\Delta_b$, then 
\bq\label{eq:pvdlpg3}
\dfrac{1}{\kappa}{\rm Lip}( \nabla  f;x,x')&=&\left\|\frac{-\langle x-x',c-a\rangle(c-a)}{\|x-x'\|\|c-a\|^2}+\frac{\langle x-x',c-b\rangle(c-b)}{\|x-x'\|\|c-b\|^2}\right\|\nn\\
&=&1-\frac{(\beta-1)^2}{1+\beta^2}.
\eq
*If $x=c+\alpha w_b\in\Delta_a$, then
\bq\label{eq:pvdlpg4}
\dfrac{1}{\kappa}{\rm Lip}( \nabla  f;x,a)&=&\frac{\|\langle x-a,a-c\rangle\|}{\|x-a\|\|a-c\|}\nn\\
&=&1-\frac{\alpha^2(\beta^2-1)^2}{(1+\beta^2)(\beta^2(1-\alpha)^2+(\beta^2-\alpha)^2)}.
\eq
*If $x=c+\alpha w_a\in\Delta_b$, then
\bq\label{eq:pvdlpg5}
\dfrac{1}{\kappa}{\rm Lip}( \nabla  f;x,b)&=&\frac{\|\langle x-b,b-c\rangle\|}{\|x-b\|\|b-c\|}\nn\\
&=&1-\frac{\alpha^2(\beta^2-1)^2}{(1+\beta^2)(\beta^2(1-\alpha)^2+(\beta^2-\alpha)^2)}.
\eq
From (\ref{eq:Axy=0}), (\ref{eq:pvdlpg1}),(\ref{eq:pvdlpg2}),(\ref{eq:pvdlpg3}),(\ref{eq:pvdlpg4}),(\ref{eq:pvdlpg5}) and Proposition \ref{pro:halAB}, we have 

\bq \label{equ:consk}
\max_{i =1,2,3} \Gamma^1(F;\omega_i)  \le k\kappa,
\eq
where $k\in (0,1)$ is a constant depending on $\beta$ and $\alpha_0$.

*Step 2: We will construct $\Omega$ open convex set, and a 1-field $G\in \cF^1(\Omega)$ such that 
\bqq
\Lip(\nabla g;\Omega)<\Gamma^1(G;\Omega).
\eqq
Using the notation $R_{x,y}$  to  denote the ray starting $x$ and passing through another point $y$. Let $A_1$ be the convex hull of $\Delta_a\cup\{R_{x_a,a}\}$, $A_2$ be the convex hull of $\Delta_a\cup \Delta_b$ and $A_3$ be the convex hull of $\Delta_b\cup\{R_{x_b,b}\}$. Let $\Omega$ be the interior of $A_1\cup A_2\cup A_3$. Then $\Omega$ is open and convex.
\begin{figure}[ht]
\begin{center}
\includegraphics[scale=0.4]{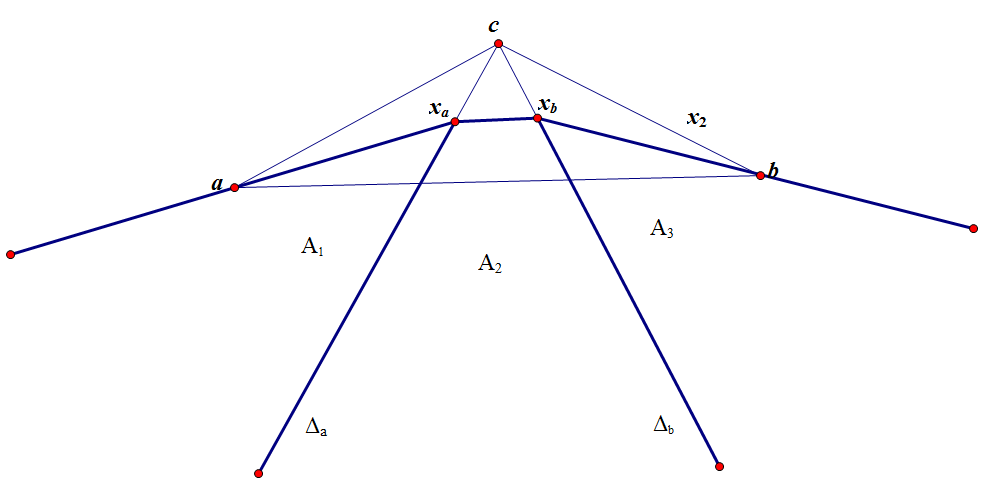}\\
\end{center}
\caption{}
\end{figure}

Let  $G_i$ be MLE of $F|_{\omega_i}$  ,   for $i\in \{1,2,3\}$.
 We define a 1-field $G$ on $\overline{\Omega}$ by $G(x)=G_i(x)$ if $x\in {A_i}$ for $i\in \{1,2,3\}$.\\
 We will prove that $G\in\cF^1(\Omega)$ and $\Lip(\nabla g;\Omega  )<\Gamma^1(G;\Omega )$.

For all $x \in \Omega$,  there exists $r>0$ such that $B(x,r) \subset \Omega$. For all $h \in B(x,r)$, we have
$$
\begin{array}{lll}
| g_{x+h}  - g_{x} - \langle D_xg ; h \rangle| & = & | G(x+h)(x+h) - G(x)(x+h)|, \\
                                               & \leq & \dfrac{1}{2} \displaystyle \max_{i =1,2,3} \Gamma^1(G_i;A_i) \|h \| ^2.
\end{array}
$$
Hence $g  \in \mathcal{C}^1(\Omega,\R)$.

Thus, by applying Lemma \ref{lem:libe} and Proposition \ref{pro:halAB} we have
\bq\label{eq:3}
\Lip(\nabla g;\Omega ) \le \displaystyle \max_{i =1,2,3} \Lip(\nabla g;{A_i})  \leq \displaystyle \max_{i =1,2,3} \Gamma^1(F;A_i) \leq k\kappa.
\eq

Therefore, by applying Corollary \ref{cor:gam<lip} we have $G\in\cF^1(\Omega)$.

On the other hand, we have
\bq\label{eq:4}
\Gamma^1(G;{\Omega)}=\Gamma^1(G;\overline{\Omega})\ge \Gamma^1(G;a,b)=\kappa .
\eq
From (\ref {eq:3}) and (\ref {eq:4}). We have 
\bqq
\Lip(\nabla g;\Omega  )\le k\kappa\le k\Gamma^1(G;\Omega )<\Gamma^1(G;\Omega ).
\eqq
\end{proof}
\begin{remark}
With the same notation as the proof of Proposition \ref{co-ex:lipgam}. There exist an open strictly convex $\Omega'$  subset of $\Omega$   such that $a,b\in \overline{\Omega'}$, and we have 
\bqq
\Lip(\nabla g;\Omega'  )\le\Lip(\nabla g;\Omega  )\le k\kappa\le k\Gamma^1(G;\Omega' )< \Gamma^1(G;\Omega' ).
\eqq
Thus in Proposition \ref{co-ex:lipgam} we can replace {\it $\Omega$ open convex} by {\it $\Omega$ open strictly convex}.

Moreover, when we let $x_a$, $x_b$ such that ${\rm dist}(x_a,[a,b])$ and ${\rm dist}(x_b,[a,b])$ converge  to $0$. 
 Then the smallest constant $k$ that is chosen satisfying (\ref{equ:consk})
converges to $\frac{\sqrt{3} }{2}$. An interesting question is that what 
is the optimal constant $c$ that is the largest constant and satisfies $\Lip(\nabla g,\Omega)\ge c\Gamma^1(F,\Omega)$ for all 
$\Omega$ open convex set and for all $F\in\cF^1(\Omega)$ ? We do not  exact value of the optimal constant $c$, but from above 
consideration and Proposition 8, we obtain
$c\in[\frac{1}{2},\frac{\sqrt{3}}{2} ]$.

\end{remark}

\begin{theorem} Let $\Omega$ be an open set in $\R^n$ and let $F\in \cF^1(\overline\Omega)$. We have 
\bqq
\Gamma^1(F;\Omega)= \max\left\{\Lip(Df;\Omega),\Gamma^1(F;\partial\Omega)\right\}.
\eqq
\end{theorem}
\begin{proof}\text{}
From \cite[Proposition 2.10]{ELG1} we know that $\Gamma^1(F;\Omega)=\Gamma^1(F;\overline\Omega)$. Thus 
\bq\label{eq:fbr}
\Gamma^1(F;\Omega)\ge \Gamma^1(F;\partial\Omega). 
\eq
Furthermore, we know that $\Gamma^1(F;\Omega)\ge \Lip(Df;\Omega)$. Therefore, 
\bqq
\Gamma^1(F;\Omega)\ge \max\left\{\Lip(Df;\Omega),\Gamma^1(F;\partial\Omega)\right\}.
\eqq
Conversely, let us turn to the proof of the opposite inequality:
\bq\label{eq:ccmlg}
\Gamma^1(F;\Omega)\le \max\left\{\Lip(Df;\Omega),\Gamma^1(F;\partial\Omega)\right\}.
\eq
Let $F _{|\partial \Omega }$ be the restriction of F to $\partial \Omega$ and let G be a MLE of $F _{|\partial \Omega }$ on $\R^n\backslash \Omega$. We have $G=F$ on $\partial\Omega$ and
\bq\label{eq:qfmr}
\Gamma^1(G;\R^n\backslash \Omega)=\Gamma^1(F;\partial \Omega).
\eq
We define 
\bqq
H(x) \triangleq \left\{ \begin{array}{ll}
F(x), \text { if }x\in\Omega,\\
G(x), \text { if }x\in\R^n\backslash \Omega.
\end{array}  \right.
\eqq
*Step 1:  We will prove that $H\in\cF^1(\R^n)$. Indeed, let $x,y\in\R^n$ ($x\ne y$). We have three cases:\\
{\it Case 1:} If $x,y\in \Omega$ ($x\ne y$) then $\Gamma^1(H;x,y)=\Gamma^1(F;x,y)\le \Gamma^1(F;\Omega)$.\\
{\it Case 2:} If $x,y\in\R^n\backslash\Omega$ ($x\ne y$) then since (\ref{eq:fbr}) and (\ref{eq:qfmr}) we have  
\bqq
\Gamma^1(H;x,y)=\Gamma^1(G;x,y)\le \Gamma^1(G,\R^n\backslash \Omega)=\Gamma^1(F;\partial \Omega)\le\Gamma^1(F,\Omega).
\eqq 
{\it Case 3:} If $x\in\Omega$ and $y\in\R^n\backslash\Omega$. Let $H_{|\{x,y\}}$ be the restriction of $H$ to $\dom (H_{|\{x,y\}})=\{x,y\}$. From \cite[Proposition 2]{ELG2}, there exists $c\in B_{1/2}(x,y)$ such that: 
\bqq
\Gamma^1(H;x,y)\le \max\{\Gamma^1(H;x,z),\Gamma^1(H;z,y)\}, ~\text{for all } z\in [x,c]\cup[y,c].
\eqq
Let $z\in \left([x,c]\cup[y,c]\right)\cap\partial \Omega$, we obtain
\bqq
\Gamma^1(H;x,y)\le \max\{\Gamma^1(H;x,z),\Gamma^1(H;z,y)\}.
\eqq
Moreover, since $x,z\in\overline \Omega$ we get
\bqq
\Gamma^1(H;x,z)&=&\Gamma^1(F;x,z)\le \Gamma^1(F;\overline \Omega)=\Gamma^1(F;\Omega),
\eqq
and since $y,z\in\R^n\backslash\Omega$ we get
\bqq
\Gamma^1(H;z,y)&=&\Gamma^1(G;z,y)\le \Gamma^1(G;\R^n\backslash\Omega)=\Gamma^1(F;\partial\Omega)\le \Gamma^1(F;\Omega).
\eqq
Therefore $\Gamma^1(H;x,y)\le \Gamma^1(F;\Omega).$\\
Combining these three cases we have $\Gamma^1(H;\R^n)\le \Gamma^1(F;\Omega)<+\infty.$ This implies that $H\in\cF^1(\R^n)$.

*Step 2: We will prove (\ref{eq:ccmlg}). Since $H\in\cF^1(\R^n)$, we have $\Gamma^1(H;\R^n) = \Lip(\nabla h;\R^n)$  by  (\cite[Proposition 2.4]{ELG1}). Thus 
 $$\Gamma^1(F;\Omega)=\Gamma^1(H;\Omega)\le \Gamma^1(H;\R^n)=\Lip(\nabla h;\R^n).$$ 
 On the other hand, $\Lip(Df;\Omega)=\Lip(\nabla h;\Omega)$ and $$\Gamma^1(F;\partial\Omega)=\Gamma^1(G;\R^n\backslash\Omega)=\Gamma^1(H;\R^n\backslash\Omega)\ge \Lip(\nabla h;\R^n\backslash\Omega).$$ 
Therefore, to prove (\ref{eq:ccmlg}), it suffices to
show that 
\bqq
\Lip(\nabla h;\R^n)\le \max\left\{\Lip(\nabla h;\Omega),\Lip(\nabla h;\R^n\backslash\Omega)\right\}.
\eqq
The final inequality is true from Lemma  \ref{lem:libe}.
\end{proof}



\section{Sup-Inf Explicit Formulas for Minimal Lipschitz Extensions for 1-Fields on $\mathbb{R}^n$}\label{chap 3}

In this section let $\Omega$ be a nonempty subset of $\cR^n$. Fix $F\in \cF^1(\Omega)$ and define $\kappa \triangleq \Gamma^1(F;\Omega)$. We will give two explicit formulas for extremal extension problem of  $F$ on $\cR^n$. 
\begin{definition}\label{Def.14}
 For any $a,b\in\Omega$ and $x\in\cR^n$, we define
\bqq
v_{a,b}&\triangleq&\frac{1}{2}(D_af+D_bf)+\frac{\kappa}{2}(b-a),\\
\alpha_{a,b}&\triangleq& 2\kappa(f_a-f_b)+\kappa\left\langle D_a f+D_b f,b-a\right\rangle-\frac{1}{2}\|D_a f-D_b f\|^2+\frac{\kappa^2}{2}\|a-b\|^2\\
&\triangleq& (\kappa A_{a,b}(F)-\frac{B_{a,b}(F)^2}{2}+\frac{\kappa^2}{2})\|a-b\|^2,\\
\beta_{a,b}(x)&\triangleq&\left\|\frac{1}{2}(D_af-D_bf)+\frac{\kappa}{2}(2x-a-b)\right\|^2,\\
\eqq
We know (see [\cite{ELG1}, Theorem 2.6]) that  $\alpha_{a,b}\ge 0$ thus we can define
\bqq
r_{a,b}(x)&\triangleq&\sqrt{\alpha_{a,b}+\beta_{a,b}(x)},\\
\Lambda_x&\triangleq&\left\{v\in\cR^n:\|v-v_{a,b}\|\le r_{a,b}(x),\forall a,b\in\Omega\right\}.
\eqq
\end{definition}

\begin{definition} For any $a\in\Omega$, $x\in\R^n$ and $v\in \Lambda_x$ we define 
\bqq
\Psi^+(F,x,a,v)&\triangleq&f_a+\frac{1}{2}\langle D_af+v,x-a\rangle+\frac{\kappa}{4}\|a-x\|^2-\frac{1}{4\kappa}\|D_af-v\|^2,\\
\Psi^-(F,x,a,v)&\triangleq&f_a+\frac{1}{2}\langle D_af+v,x-a\rangle-\frac{\kappa}{4}\|a-x\|^2+\frac{1}{4\kappa}\|D_af-v\|^2,\\
u^+(x)&\triangleq&\sup\limits_{v\in\Lambda_x}\inf\limits_{a\in\Omega} \Psi^+(F,x,a,v),\\
\eqq
\end{definition}
An important part of the proof of the  [Theorem 2.6, \cite{ELG1}] shows that  $\Lambda_x$ is nonempty  for all $x \in \R^n$. Since   $\Lambda_x$ is nonempty and compact, and  the map $v\longmapsto \Psi^+(F,x,a,v)$ 
is continuous, the map    $u^+$ is well defined.

\begin{proposition}\label{pro:xddn} Fix $x\in\R^n$. Then there exists a unique element $v_x^+\in \Lambda_x$ such that 
$$u^+(x)=\inf\limits_{a\in\Omega} \Psi^+(F,x,a,v_x^+).$$
\end{proposition}
\begin{proof}
Since $\Lambda_{x}$  is compact and nonempty, there exists  $v_x^+\in \Lambda_{x}$ such that 
\bqq
u^+(x)=\inf\limits_{a\in\Omega} \Psi^+(F,x,a,v_x^+).
\eqq
We will prove that $v_x^+$  is uniquely determined. Indeed, for any $a\in\Omega$ we define
\bqq
g_a(v)=\Psi^+(F,x,a,v),\text{ for }v\in\Lambda_x.
\eqq
Then for any $t\in(0,1)$ and  
 $(v_1,v_2) \in \R^n\times\R^n$ 
\bqq
g_a(tv_1+(1-t)v_2)= tg_a(v_1)+(1-t)g_a(v_2)+\frac{1}{4\kappa}t(1-t)\|v_1-v_2\|^2.
\eqq
Thus $g_a$ is strictly concave.
If we define $g(v)=\inf\limits_{a\in\Omega}g_a(v)$ for $v\in\Lambda_x$ then for any $t\in(0,1)$ and  $(v_1,v_2) \in \R^n\times\R^n$ 
\bqq
g(tv_1+(1-t)v_2) \geq tg(v_1)+(1-t)g(v_2)+\frac{1}{4\kappa}t(1-t)\|v_1-v_2\|^2.
\eqq
Thus $g$ is also strictly concave. \\To prove $v_x^+$ is uniquely determined, we need to prove that if $g(v)=g(v_x^+)$  then $v=v_x^+$, where $v \in\Lambda_x$. Assume by contradiction there exists $v\in \Lambda_x$ such that $v\ne v_x^+$ and $g(v)=g(v_x^+)$. Since $\Lambda_x$ is a convex set we have $tv+(1-t)v_x^+\in\Lambda_x$ for $t\in(0,1)$. Thus $g(tv+(1-t)v_x^+)> tg(v)+(1-t)g(v_x^+)=g(v_x^+)$, which is a contradiction to $g(v_x^+)=\sup\limits_{v\in\Lambda_x}g(v)$.
\end{proof}

The previous proposition  allows to define   the following one-field

\begin{definition}\label{uplus}
\begin{align}
U^+: \,  x \in \R^n  \mapsto U^+(x)(a) \triangleq u^+(x) + \langle D_xu^+ ; a - x\rangle, a \in \R^n.
\end{align}
with

$$
u^+(x)  \triangleq \max\limits_{v\in\Lambda_x}\inf\limits_{a\in\Omega} \Psi^+(F,x,a,v),~~~ D_xu^+\triangleq \mbox{arg}\displaystyle\max_{v\in\Lambda_x}\inf\limits_{a\in\Omega} \Psi^+(F,x,a,v).
$$
\end{definition}

Using the proof of [Theorem 2.6, \cite{ELG1}], we can easily show the following proposition

\begin{proposition}\label{pro:the2.61}
Let $x_0\in\R^n$ and define $\Omega_1=\Omega\cup \{x_0\}$. Let $U$ 
an extension of $F$ on $\Omega_1$. 
 Then the following conditions are equivalent\\
$(i)$ $U$ is a MLE of $F$ on $\Omega_1.$ \\
$(ii)$ $$\sup\limits_{a\in\Omega}\Psi^-(F,x,a,D_{x}u)\le u(x) \le \inf\limits_{a\in\Omega}\Psi^+(F,x,a,D_{x}u),~~ \forall x\in\Omega_1.$$

Furthermore
$$\left[\sup\limits_{a\in\Omega}\Psi^-(F,x,a,D_{x}u)\le \inf\limits_{a\in\Omega}\Psi^+(F,x,a,D_{x}u)\right]\Leftrightarrow \left[D_{x}u\in\Lambda_{x}\right],~~ \forall x\in\Omega_1.$$
\end{proposition}
\begin{corollary} \label{co:dlambda}Let $\Omega_1$ be a subset of $\R^n$ such that $\Omega\subset\Omega_1$. Let $G$ be a MLE of $F$ on $\Omega_1$. For all $x\in\Omega_1$, we have $D_xg\in\Lambda_x$ and 
$$\sup\limits_{a\in\Omega}\Psi^-(F,x,a,D_{x}g)\le g(x) \le \inf\limits_{a\in\Omega}\Psi^+(F,x,a,D_{x}g)\le u^+(x), ~~\forall x\in\Omega_1.$$
\begin{proof}
The proof is immediate from Proposition \ref{pro:the2.61}.
\end{proof}

\end{corollary}

\begin{theorem}\label{overUplus}
 The $1$-field $U^+$ is the unique over extremal extension of $F$. 
\end{theorem}
\begin{proof}
 Applying Theorem \ref{theo:tqwells}, there exists $W^+$ an over extremal extension of $F$ on $\R^n$.  Let $w^+$  be the canonical associates to $W^+$. 
We will prove $U^+=W^+$ on $\R^n$.

*Step 1: Let $x\in\Omega$. Since  $W^+$ is an extension of $F$ we have $W^+(x)=F(x)$.
Noting that $\Lambda_x$ has a unique element to be $D_xf$ (since $x\in\Omega$ and from the definition of $\Lambda_x$).  So that $D_xu^+ = D_xf$ and
\bqq
u^+(x)&=&\inf\limits_{a\in\Omega} \Psi^+(F,x,a,D_xf).
\eqq
From Proposition \ref{pro:the2.61} we have 
\bqq
\Psi^+(F,x,a,D_xf)\ge f(x), \text{ for any }a\in\Omega.
\eqq
Furthermore, when $a=x$ we have $\Psi^+(F,x,x,D_xf)=f(x)$. Therefore $u^+(x)= f(x)$. \\
Conclusion for all $x \in \Omega$,   $U^+(x) =W^+(x)$.

*Step 2:  Let  $x\in \R^n\backslash\Omega$. We first prove that $u^+(x)\ge w^+(x)$. Since $W^+$ is a MLE of $F$ on $\R^n$, we can apply Proposition \ref{pro:the2.61} to obtain $D_{x}w\in\Lambda_{x}$ and 
\bqq
w^+(x)\le\inf\limits_{a\in\Omega} \Psi^+(F,x,a,D_{x}w)\le u^+(x).
\eqq

Conversely, we will prove that $u^+(x)\le w^+(x)$. Applying Proposition \ref{pro:xddn}, $D_xu^+$ is the unique element in $\Lambda_{x}$ such that $u^+(x)=\inf\limits_{a\in\Omega} \Psi^+(F,x,a,D_{x}u^+)$. 

 We define the 1-field $G$ of domain  $\Omega\cup\{x\}$ as $$G(y)  \triangleq F(y),~ y \in \Omega ~{\rm  and }~ G(x) \triangleq  U^+(x)$$
Since $D_{x}g=D_{x}u^+\in\Lambda_{x}$, we can apply Proposition \ref{pro:the2.61}  
 to have

\bq\label{eq:dkcpr}
g(x)= u^+(x)=\inf\limits_{a\in\Omega} \Psi^+(F,x,a,D_{x}g)\ge \sup\limits_{a\in\Omega}\Psi^-(F,x,a,D_{x}g).
\eq
From (\ref{eq:dkcpr}) and Proposition \ref{pro:the2.61}, we have $G$ to be a MLE of $F$ on $\Omega\cup\{x\}$.

By applying [Theorem 2.6, \cite{ELG1}] 
 there exists $\tilde{G}$ to be a MLE of $G$ on $\R^n$.

 Since ${\rm dom}(F) \subset  {\rm dom}(G)$,    $\tilde{G}$ is also a MLE of $F$ on $\R^n$. Since $W^+$ is over extremal extension of $F$ on $\R^n$, we have $\tilde{g}(x)\le w^+(x)$ for all $x\in\R^n$. Thus $u^+(x)=g(x)=\tilde{g}(x)\le w^+(x)$. 
Combining this with $w^+(x)\le u^+(x)$ we have $u^+(x)= w^+(x)$.

Finally, using Proposition \ref{pro:the2.61} and the previous equality we have
$$
u^+(x) =  w^+(x)\le\inf\limits_{a\in\Omega}\Psi^+(F,x,a,D_x w^+)\le u^+(x)=\inf\limits_{a\in\Omega}\Psi^+(F,x,a,D_{x} u^+).
$$
Thus we obtain the following equality
$$
\inf\limits_{a\in\Omega}\Psi^+(F,x,a,D_x w^+)=\inf\limits_{a\in\Omega}\Psi^+(F,x,a,D_{x} u^+).
$$
Therefore $D_x w^+=D_{x} u^+$ by Proposition \ref{pro:xddn}.

Conclusion for all $x \in \R^n$,   $U^+(x) =W^+(x)$.

The uniqueness of an over extremal extension of $F$  arises because  $W^+$ be to any   over extremal extension of $F$.

\end{proof}

Thanks to result of Theorem \ref{theo:tqwells}. Indeed, this result allows to define an under extremal extension of $F$. That is

\begin{definition} For any $a\in\Omega$, $x\in\R^n$ and $v\in \Lambda_x$ we define 
\bqq
u^-(x)&\triangleq&\inf\limits_{v\in\Lambda_x}\sup\limits_{a\in\Omega} \Psi^-(F,x,a,v).
\eqq
\end{definition}
Using the strict convexity of the map $v \longmapsto \Psi^-(F,x,a,v)$ and the compacity of $\Lambda_x$ as in the proof of proposition \ref{pro:xddn} (replacing concavity by convexity) we obtain the following proposition

\begin{proposition}\label{promoins:xddn} Let $x\in\R^n$. Then there exists a unique element $v_x^-\in \Lambda_x$ such that $u^-(x)=\sup\limits_{a\in\Omega} \Psi^-(F,x,a,v_x^-)$.
\end{proposition}
This allows us to define the following $1$-field

\begin{definition}\label{u_moins}
\begin{align}
U^-: \,  x \in \R^n  \mapsto U^-(x)(a) \triangleq u^-(x) + \langle D_xu^- ; a - x\rangle,~ a \in \R^n.
\end{align}
with

$$
u^-(x)  \triangleq \min\limits_{v\in\Lambda_x}\sup \limits_{a\in\Omega} \Psi^-(F,x,a,v),~~~ D_xu^-\triangleq \mbox{arg}\displaystyle\min_{v\in\Lambda_x}\sup\limits_{a\in\Omega} \Psi^-(F,x,a,v).
$$
\end{definition}

\begin{theorem}
 The $1$-field $U^-$ is the unique under extremal extension of $F$. 
\end{theorem}
\begin{proof}
 Using theorem \ref{theo:tqwells} and the Proposition (\ref{promoins:xddn}) the proof using similar arguments as in the proof of theorem \ref{overUplus}
\end{proof}

In conclusion we have the following corollary

\begin{corollary} For all minimal Lischitz extension $G$ of $F$ we have
$$
  u^-(x) \leq g(x) \leq u^+(x),  ~\forall x \in  \R^n.
$$
 
\end{corollary}

\section{Sup-Inf Explicit Formulas for Minimal Lipschitz Extensions for  function from  $\mathbb{R}^m$ to   $\mathbb{R}^n$  }\label{chap 4}

Now, we propose to use the previous results to produce formulas comparable to those Bauschke and Wang have found see \cite{Ba}.

Let us define first $\mathcal{Q}_0$  as the problem of the minimum extension for Lipschitzian functions second  $\mathcal{Q}_1$  as the problem of the minimum extension for $1$-fields. Curiously, we will show that the problem  $\mathcal{Q}_0$ is a sub-problem of the problem   $\mathcal{Q}_1$. As a consequence, we obtain two explicit formulas that solve the problem  $\mathcal{Q}_0$. \\
More specifically, fix $n,m \in  \N^*$ and $\omega \subset \R^m$. Let $u$ be a function  from  $\omega $ maps to  $\mathbb{R}^n$. Suppose ${\rm Lip}(u;\omega) < +\infty$ and define $l \triangleq {\rm Lip}(u;\omega)$.

Let us define $$\Omega \triangleq \{(x,0)\in \R^m \times \R^n : x \in \omega\}$$
 A current element  of $\R^{m+n}$ will denote by  $x:=(x^{(m)}, x^{(n)}) \in \R^{m+n}$, with $x^{(m)} \in \R^m$ and $x^{(n)} \in \R^n$.

Now we will define a $1$-field associated with $u$ denote by $F$ from $\Omega \subset \R^{n+m}$ maps to $\mathcal{P}^1(\R^{n+m},\R)$ as the following

\begin{equation}\label{KirsTheo}
f(x,0) \triangleq 0,~~{\rm and }~~D_{(x,0)}f \triangleq   (0,u(x)),~~{\rm for~all~} x \in \omega.
\end{equation}

Let $a,b \in \omega$, with $a \neq b$. Observing that
$$
f(a,0)=f(b,0)=0, ~{\rm and }~\langle D_{(a,0)}f+D_{(b,0)}f , (b-a,0)\rangle=0,
$$
 and applying Proposition \ref{pro:halAB} we have

$$
 \Gamma^1(F,(a,0),(b,0)) = \frac{\|D_{(a,0)}f-D_{(b,0)}f  \|}{\|b-a\|}.
$$
Therefore
\begin{equation}\label{Kir-1}
\Gamma^1(F,\Omega ) = \Lip(u,\omega).
\end{equation}

Let $G$ be an minimal Lipschitz extension of $F$. We have  $G   \in\cF^1(\R^{m+n}) $ and 
\begin{equation}\label{Kir-2}
  \Gamma^1(F,\Omega ) = \Gamma^1(G,\R^{m+n} ).
\end{equation}

Using \cite[Proposition 2.4]{ELG1} we have 
\begin{equation}\label{Kir-3}
\Gamma^1(G;\R^{m+n}) = {\rm Lip} (Dg;\R^{m+n}).
\end{equation}

Now we define the map $\tilde{u}$ from $\R^m$ to $\R^n$ as following
\begin{equation}\label{Kir-4}
 \tilde{u}(x) \triangleq  (D_{(x,0)}g)^{(n)},~ x \in \R^m.
\end{equation}
We will show that $\tilde{u}$ is a minimal Lipschitz extension of $ u $.

First, let $x \in \omega$. Since $G$ is an extension of $F$ and by construction of $F$ we have
$$
  \tilde{u}(x) = (D_{(x,0)}g)^{(n)} = u(x).
$$
Thus $ \tilde{u}$ is an extension of $u$.\\
Second, let $x,y \in \R^m$ with $x \neq y$. Using (\ref{Kir-1}), (\ref{Kir-2})  and (\ref{Kir-3}) we have
\begin{equation}\label{26}
 \left. \begin{array}{llll}
{\rm Lip}(\tilde{u};x,y)     &   = & \frac{  \|(D_{(x,0)}g)^{(n)}-(D_{(y,0)}g)^{(n)}\|}{\|x - y\|},\\
                                       & \leq &    \frac{ \| D_{(x,0)}g  -D_{(y,0)}g \| }{\|x - y\|}, \\
                                       & \leq   &     \Gamma^1(G;\R^{m+n}),  \\
                                     &   =      &   \Lip(u,\omega).
\end{array}\right. 
\end{equation}
Conclusion,  $\tilde{u}$ is an minimal Lipchitz extension of $u$. Therefore,  we obtain another proof of Kirsbraun's theorem (see \cite{Kirs}).

\begin{theorem}\label{TheoKirs} 
 Let $u$ be a function  from  $\omega \subset \R^m$  to  $\mathbb{R}^n$. Suppose $u$ Lipschitzian. Let $F$ be the $1$-field defined by the formula (\ref{KirsTheo}) . Let $G$ be  any minimal Lipschitz extension of $F$. Then the extension $\tilde{u}$ define by the formula (\ref{Kir-4})   is a minimal Lipschitz extension of $u$.
\end{theorem}
By replacing $G$ in Theorem \ref{TheoKirs}, by $U^-$ and $U^+$, we have two formulas which resolve the problem $\mathcal{Q}_0$.
Now describe these formulas which resolve the problem  $\mathcal{Q}_0$.



Let $a,b\in \omega$ and $x\in \cR^{m}$ we define
\bqq
v_{a,b}& \triangleq &\left(\frac{l}{2}(b-a),\frac{1}{2}(u(a)+u(b))\right),\\
\alpha_{a,b}& \triangleq &-\frac{1}{2}\|u(a)-u(b)\|^2+\frac{l^2}{2}\|a-b\|,\\
\beta_{a,b}(z)& \triangleq &\|lx-a-b\|^2+\|\frac{1}{2}(u(a)-u(b))\|^2,\\
r_{a,b}(x)& \triangleq &\sqrt{\alpha_{a,b}+\beta_{a,b}(x)},\\
\Lambda_x& \triangleq &\{v\in\cR^{m+n}: \|v-v_{a,b}\|\le r_{a,b}(x),\forall a,b\in \omega\}.
\eqq

 For  $v\in \R^{m+n}$, we define 
\bq  \label{Over-Kir1}
\Phi^+(u,x,a,v)& \triangleq &\frac{1}{2}\langle v^{(m)},x-a\rangle+\frac{l}{4}\|a-x\|^2-\frac{1}{4l}(\|v^{(m)}\|^2+\|u(a)-v^{(n)}\|^2),
\eq
\bq  \label{Under-Kir1}
\Phi^-(u,x,a,v)& \triangleq &\frac{1}{2}\langle v^{(m)},x-a\rangle-\frac{l}{4}\|a-x\|^2+\frac{1}{4l}(\|v^{(m)}\|^2+\|u(a)-v^{(n)}\|^2).
\eq
Now using the previous notation, we define two maps from $\cR^m$ to $\cR^n$ as following
\begin{equation} \label{Over-Kir2}
 k^+(x)  \triangleq   (\mbox{arg}\displaystyle\max_{v\in\Lambda_x}\inf\limits_{a\in\omega} \Phi^+(u,x,a,v))^{(n)},~ x\in \cR^m,
\end{equation}
and
\begin{equation} \label{Under-Kir2}
  k^-(x) \triangleq    (\mbox{arg}\displaystyle\min_{v\in\Lambda_x}\sup\limits_{a\in\omega} \Phi^-(u,x,a,v))^{(n)},~x\in \cR^m.
\end{equation}

\begin{theorem} 
The maps $k^+$ and $k^-$ define by the formulas   (\ref{Over-Kir2}) and  (\ref{Under-Kir2}) are minimal Lipschitz extensions of $u$.
\end{theorem}

\begin{comment}
If $\omega$ is finite using the previous transformation $u \longrightarrow F$ the Wells explicite construction of $u^+$ or $u^-$ allows to compute $u^+$ and $u^-$.


We know that the proof of Kirszbraun-Valentine's theorem and the proof of [Theorem 2.6, \cite{ELG1}] use  Zorn's lemma.
Noticing that, like \cite{Ba}, the proof showing that  $ \Lambda_x$  is not empty and the proof of Wells' Theorem  \ref{theo:W+} do not use  Zorn's lemma. 



\end{comment}






\section{Absolutely Minimal Lipschitz Extensions}\label{chap 5}
\subsection{Finite domain}\label{sec:trfini}
Let $A=\{p_1,...,p_m\}$ be a finite subset of $\mathbb{R}^n$ ($n\ge 2$) and let $F\in\cF^1(A)$. Fix $\kappa=\Gamma^1(F;A)$. In this section, we  prove that the functions $W^+(F,A,\kappa)$ and $W^-(F,A,\kappa)$ defined in Appendix \ref{con:well+} and \ref{con:well-} are AMLEs of $F$ on $\R^n$.

\begin{theorem}\label{theo: wwamle} $W^+(F,A,\kappa)$ and $W^-(F,A,\kappa)$ are AMLEs of $F$ on $\R^n$.
\end{theorem}
\begin{proof}
For brevity let us denote $W^+(F,A,\kappa)$ by $W^+$. We prove that $W^+$ is an AMLE of $F$ on $\R^n$.  \\
Since Corollary \ref{cor:well+}, we have $W^+$ to be an MLE of $F$ on $\R^n$. \\
Let $V$ be a bounded open satisfying $\bar{V} \subset \R^n\setminus A$. We need to prove that $\Gamma^1(W^+;V)=\Gamma^1(W^+;\partial V)$. Indeed, the inequality $\Gamma^1(W^+;V)\ge\Gamma^1(W^+;\partial V)$ is clearly true, so that we only need to prove that $\Gamma^1(W^+;\partial V)\ge\Gamma^1(W^+;V)$. 
Because $\kappa=\Gamma^1(W^+;\cR^n)\ge \Gamma^1(W^+;V)$ and $\Gamma^1(W^+;\partial V)\ge \Lip(\nabla w^+;\partial V)$ (from Proposition \ref{pro:halAB}), it is enough to show that $\Lip(\nabla w^+,\partial V)\ge \kappa$.\\
Applying Proposition \ref{pr:cw+2} we have $\bigcup\limits_{S \in K^+} {{T^+_S}}  = {\mathbb{R}^n}$ ($K^+$ and $T^+_S$ are defined in Appendix \ref{con:well+}). Since $A$ has finite elements, we have $K^+$ has finite elements. On the other hand, $\partial V$ has infinite elements. Therefore there exist $S\in K^+$ and $x_0,y_0\in \partial V$ such that $x_0,y_0\in T^+_S$. Applying \cite[Lemma 21]{Well1} we have
\bqq
\|\nabla w^+(x_0)-\nabla w^+(y_0)\|=\kappa\|x_0-y_0\|.
\eqq
Thus
\bqq
\Lip(\nabla w^+,\partial V)\ge \frac{\|\nabla w^+(x_0)-\nabla w^+(y_0)\|}{\|x_0-y_0\|}=\kappa.
\eqq
Therefore  $W^+$ is an AMLE of $F$ on $\R^n$.\\
The proof for $W^-$ is similar.
\end{proof}

\begin{corollary}\label{Coro.infinity} There exist infinity solutions AMLE of $F$ on $\cR^n$.
\end{corollary}
\begin{proof}
From the definition of $W^+(F,A,\kappa)$ and $W^-(F,A,\kappa)$, we have $W^+(F,A,\kappa)\ne W^-(F,A,\kappa)$. 

Let $x_0\in\cR^n\backslash A$ such that $W^+(F,A,\kappa)(x_0)\ne W^-(F,A,\kappa)(x_0)$. Noting that  $$W_\tau=\tau W^+(F,A,\kappa)+(1-\tau) W^-(F,A,\kappa)$$ is MLE of $F$ on $\cR^n$, for any $\tau\in[0,1]$. 

Thus there exist infinity solution MLE of $F$ on $A\cup\{x_0\}$.

Let G be a MLE of $F$ on $A\cup\{x_0\}$. By the same argument as the proof of Theorem \ref{theo: wwamle}, we have $W^+(G,A\cup\{x_0\},\kappa)$ is the AMLE of $F$ on $\cR^n$. 

Therefore, we have infinity solutions AMLE of $F$ on $\cR^n$.
\end{proof}

\subsection{Infinite domain }\label{sec:wellctq}
Let $\Omega$ be a nonempty subset of $\cR^n$ and let $F\in\cF^1(\Omega)$. Fix $\kappa=\Gamma^1(F;\Omega)$. From Section \ref{sec:trfini}, we know that if $\Omega$ is a finite set then the functions $W^+(F,\Omega,\kappa)$ and $W^-(F,\Omega,\kappa)$ defined in Appendix \ref{con:well+} and \ref{con:well-} are AMLEs of $F$ on $\R^n$ ($n\ge 2$). So that, we hope that in general case when $\Omega$ is an infinite set, the functions $W^+(F,\Omega)$ and $W^-(F,\Omega)$ defined in Appendix \ref{sec:casedoinf} are also AMLEs of $F$ on $\R^n$. Unfortunately, this is not true. We give an example that $\Omega$ is an infinite set and neither $W^+(F,\Omega)$ nor $W^-(F,\Omega)$ is AMLE of $F$ on $\R^n$.
\begin{proposition} \label{pro:mdpvdqt} We consider in $\cR^2$. Let $\Omega_1=\partial B(0;1)$ and $\Omega_2=\partial B(0;2)$. Let $\Omega=\Omega_1\cup\Omega_2$. Let $F\in\cF^1(\Omega)$ such that $f_x=0$ for all $x$ in $\Omega_1$, $f_x=1$ for all $x$ in $\Omega_2$, and $D_xf=0$ for all $x$ in $\Omega$.
Then neither $W^+(F,\Omega)$ nor $W^-(F,\Omega)$ is AMLE of $F$ on $\R^2$.
\end{proposition}
\begin{proof}
For brevity let us denote $W^+(F,\Omega)$ by $W^+$. We will prove that $W^+$ is not an AMLE of $F$ on $\R^n$. To do this, we need to find an open set $V\subset\subset\cR^2/\Omega$ such that $\Gamma^1(W^+;V)\ne \Gamma^1(W^+;\partial V)$.\\
Let $V=\{x\in\R^2:\|x\|< 3/4\}\subset\subset\cR^2/\Omega$, we will prove $\Gamma^1(W^+;V)\ne \Gamma^1(W^+;\partial V)$. 
\begin{figure}[ht]
\begin{center}
\includegraphics[scale=0.4]{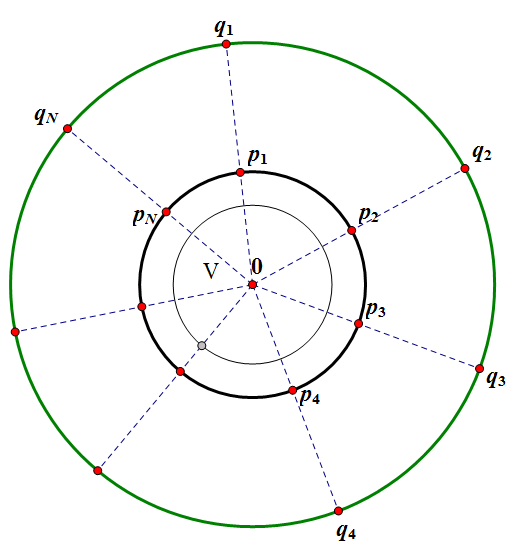}\\
\end{center}
\caption{}
\end{figure}\\
Applying Proposition \ref{pro:halAB} we have $\kappa=\Gamma^1(F;\Omega)=4$. \\
Let $\mathcal{A}$ be the set of all finite sets $A_N$, where $A_N=\{p_1,...,p_N\}\cup\{q_1,...,q_N\}$ satisfies $\{p_1,p_2,...p_N\}\subset\Omega_1$, $\{q_1,q_2,...,q_N\}\subset\Omega_2$ and $p_i\in [0,q_i]$ for all $i\in {1,...,N}$.\\
Let $A_N\in \mathcal{A}$. For brevity let us denote the functions $W^+(F,A_N,\kappa)$ defined in Appendix \ref{con:well+} by $W_{A_N}^+$.
\\ Applying Proposition \ref{pro:halAB} we have $\kappa_N=\Gamma^1(W_{A_N}^+;A_N)=4=\kappa$.\\ We have $[\frac{0+p_i}{2},\frac{p_i+q_i}{2}]\subset T^+_{p_i}$, for all $i\in \{1,...,N\}$, where $T^+_{p_i}$ is defined in Appendix \ref{con:well+} with the corresponding definition for the finite set $A=A_N$. From the definition of $w^+_{A_N}$, we have: $$w_{A_N}^+(x)=\frac{\kappa}{2}(x-p_i)^2=\frac{\kappa}{2}d^2(x;\partial \Omega_1),$$ for all $x\in[\frac{p_i}{2},\frac{p_i+q_i}{2}]$ and $i\in \{1,...,N\}$.\\
We will prove $w^+(x)=\frac{\kappa}{2}d^2(x;\partial \Omega_1)$ for all $x\in\{1/2\le \|x\|\le 3/2\}$. 
Indeed, for any $x_0\in\{x:1/2\le \|x\|\le 3/2\}$, there exist $A_N\in \mathcal{A}$ such that $x_0\in[0,q_1]$. From the definition of $w^+$, we have $w^+(x_0)\le w_{A_N}^+(x_0)=\frac{\kappa}{2}d^2(x_0;\partial \Omega_1)$. Conversely, we call $\mathcal{P}$ to be the set of all finite subsets of $\Omega$. Then for any $P\in \mathcal{P}$, there exist $A_N\in \mathcal{A}$ such that $P\subset A_N$ and $x_0\in [0,q_1]$. Applying Theorem \ref{theo:W+}, we have $w^+_P(x_0)\ge w^+_{A_N}(x_0)=\frac{\kappa}{2}d^2(x_0;\partial \Omega_1)$. Hence $w^+(x_0)=\inf\limits_{P\in\mathcal{P}}w^+_P(x_0)\ge \frac{\kappa}{2}d^2(x_0,\partial \Omega_1)$. Therefore $w^+(x)=\frac{\kappa}{2}d^2(x;\partial \Omega_1)$ for all $x\in\{1/2\le \|x\|\le 3/2\}$.\\
Thus $$w^+(x)=\frac{\kappa}{2}d^2(x;\partial \Omega_1)=\frac{\kappa}{2}(\|x\|-1)^2,$$ for all $x\in V\cap\{\frac{1}{2}< \|x\|< \frac{3}{2}\}=\{\frac{1}{2}< \|x\|< \frac{3}{4}\}$. \\
Let $x_0,y_0\in \{\frac{1}{2}< \|x\|< \frac{3}{4}\}$ such that $x_0\in [0,y_0]$, we have 
\bqq
\|\nabla w^+(x_0)-\nabla w^+(y_0)\|=\left\|\kappa(\|x_0\|-1)\frac{x_0}{\|x_0\|}-\kappa(\|y_0\|-1)\frac{y_0}{\|y_0\|}\right\|=\kappa\|x_0-y_0\|.
\eqq
Hence $\kappa\ge \Gamma^1(W^+;\cR^n)\ge \Gamma^1(W^+;V)\ge \Lip(\nabla w^+;V)\ge \kappa$, and hence
\bq\label{pt:uvd}
\Gamma^1(W^+;V)=\kappa=4.
\eq
On the other hand, for all $x,y\in \partial V$ we have $\|x\|=\|y\|=3/4$, so that 
\bqq
\|\nabla w^+(x)-\nabla w^+(y)\|=\left\|\kappa(\|x\|-1)\frac{x}{\|x\|}-\kappa(\|y\|-1)\frac{y}{\|y\|}\right\|=\frac{\kappa}{3}\|x-y\|.
\eqq
Hence $$B_{x,y}(W^+)= \frac{\left\| {\nabla w^+(x)-\nabla w^+(y)} \right\|}{\left\| {x-y} \right\|}=\frac{\kappa}{3}.$$\\
Moreover, since $w^+(x)=w^+(y)$ and ($\nabla w^+(x)+\nabla w^+(y)$) is perpendicular to $(y-x)$, we have
\bqq
A_{x,y}(W^+) =\frac{2(w^+(x)-w^+(y))+\left\langle{\nabla w^+(x)+\nabla w^+(y),y-x}\right\rangle}{\left\| {x-y} \right\|^2}=0.
\eqq
Applying Proposition \ref{pro:halAB}, we have 
\bq\label{pt:uvd1}
\Gamma^1(W^+;\partial V)=\sup\limits_{x,y\in\partial V}\left(\sqrt{A_{a,b}^2+B_{a,b}^2}+\left| A_{a,b}\right|\right)=\frac{\kappa}{3}=\frac{4}{3}.
\eq
From (\ref{pt:uvd}) and (\ref{pt:uvd1}) we have $\Gamma^1(W^+;V)\ne\Gamma^1(W^+;\partial V)$. And therefore $W^+$ is not an AMLE of $F$ on $\R^n$.\\
The proof for $W^-$ is similar.
\end{proof}
\begin{remark}
With the same notation as in Proposition  \ref{pro:mdpvdqt}. For brevity let us denote $w^+=w^+(F,\Omega)$ and $w^-=w^-(F,\Omega)$, by computing directly from the definition, we have 
\bqq
w^+(x)&=&-w^-(x)=1-\frac{\kappa}{2}x^2,\text{  }\forall x\in\{x\in\cR^2:0\le\|x\|\le\frac{1}{2}\},\\ 
w^+(x)&=&-w^-(x)=\frac{\kappa}{2}d^2(x,\partial \Omega_1),\text{  }\forall x\in\{x\in\cR^2:\frac{1}{2}\le\|x\|\le 1\},\\ 
w^+(x)&=&w^-(x)=\frac{\kappa}{2}d^2(x,\partial \Omega_1),\text{  }\forall x\in\{x\in\cR^2:1\le\|x\|\le \frac{3}{2}\},\\ 
w^+(x)&=&w^-(x)=1-\frac{\kappa}{2}d^2(x,\partial \Omega_2),\text{  }\forall x\in\{x\in\cR^2:\frac{3}{2}\le\|x\|\le 2\},\\ 
w^+(x)&=&1+\frac{\kappa}{2}d^2(x,\partial \Omega_2) ,\text{  }\forall x\in\{x\in\cR^2:2\le\|x\|\},\\ 
w^-(x)&=& 1-\frac{\kappa}{2}d^2(x,\partial \Omega_2),\text{  }\forall x\in\{x\in\cR^2:2\le\|x\|\},
\eqq
where $\kappa=\Gamma^1(F;\Omega)=4$.
\\
We see that $w=\frac{w^+ + w^-}{2}$ is an AMLE of $F$ on $\cR^2$ (although $w^+$ and $w^-$ are not AMLEs of $F$ on $\R^2$) and $w\notin \mathcal{C}^2(\cR^2,\R)$. Moreover, all MLEs of $F$ coincide on $\{x\in\cR^2:1\le\|x\|\le 2\}$ (because $w^+=w^-$ on $\{x\in\cR^2:1\le\|x\|\le 2\}$). 
\end{remark}

From Proposition \ref{pro:mdpvdqt}, we know that $W^+(F,\Omega)$ is not an AMLE of $F$ on $\R^n$ in general case. But in some case, we have $W^+(F,\Omega)$ to be an AMLE of $F$ on $\R^n$. We give an example:

\begin{proposition}Let $\Omega$ be a nonempty subset of $\cR^n$ ($n\ge 2$). Let $F\in\cF^1(\Omega)$ such that $\widetilde \Omega=\{p-\frac{D_pf}{\kappa}:p\in \Omega\}$ is a subset of an $(n-1)$-dimensional hyperplane $H$, where $\kappa=\Gamma^1(F,\Omega)$. Then $W^+(F,\Omega)$ is an AMLE of $F$ on $\R^n$.
\end{proposition}
\begin{proof}
For brevity let us denote $W^+(F,\Omega)$ by $W^+$. We prove that $W^+$ is an AMLE of $F$ on $\R^n$. Put $\kappa=\Gamma^1(F;\Omega)$. \\
From Theorem \ref{theo:tqwells}, we have $W^+$ to be an MLE of $F$ on $\R^n$. \\
Let $V\subset \R^n\backslash A$. We need to prove  $\Gamma^1(W^+;V)=\Gamma^1(W^+;\partial V)$. Indeed, the inequality $\Gamma^1(W^+;V)\ge\Gamma^1(W^+;\partial V)$ is clear, so that we only need to prove that $\Gamma^1(W^+;\partial V)\ge\Gamma^1(W^+;V)$. We have $\kappa=\Gamma^1(W^+;\cR^n)\ge \Gamma^1(W^+;V)\ge\Gamma^1(W^+;\partial V)\ge \Lip(\nabla w^+,\partial V)$, so that it suffices to show that $\Lip(\nabla w^+,\partial V)\ge \kappa$. \\
Let $x_0,y_0\in \partial V$, $x_0\ne y_0$ such that $(x_0-y_0)$ is perpendicular to the hyperplane $H$. 
\begin{figure}[ht]
\begin{center}
\includegraphics[scale=0.5]{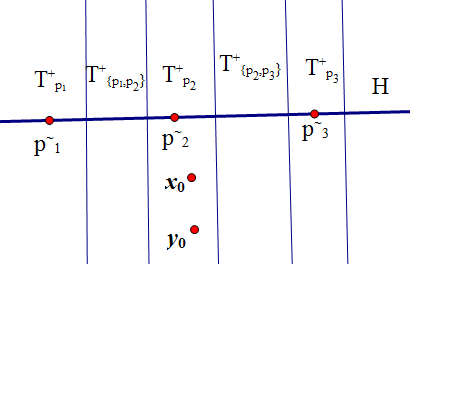}\\
\end{center}
\caption{}
\end{figure}\\
Let $\mathcal{P}$ be the set of all finite subsets of $\Omega$. For any $P\in\cP$, we have the corresponding function $W^+(F,P,\kappa)$ (defined in Appendix \ref{con:well+}) or $W^+_P$ for short. We define $K^+$ and $T^+_S$ for $S\in K$ as in Appendix \ref{con:well+} with the corresponding definition for the finite set $A=P$. Put $\kappa_P=\Gamma^1(W^+_P;\cR^n)$. Since $\widetilde \Omega=\{p-\frac{D_pf}{\kappa}:p\in \Omega\}$ is a subset of $H$ and $(x_0-y_0)$ is perpendicular to the hyperplane $H$, there exist $S\in K^+$ such that $x_0,y_0\in T^+_S$. Applying \cite[Lemma 21]{Well1}, we have
\bqq
\|\nabla w_P^+(x_0)-\nabla w_P^+(y_0)\|=\kappa_P\|x_0-y_0\|.
\eqq
Moreover, for any $\eps>0$, there exists $P\in\cP$ such that 
$\|\nabla w_P^+(x_0)-\nabla w^+(x_0)\|\le \eps$, $\|\nabla w_P^+(y_0)-\nabla w^+(y_0)\|\le \eps$ (by [\cite{Well1}, Proposition 3]) and $\kappa_P>\kappa-\eps$. Therefore 
\bqq
\Lip(\nabla w,\partial V)&\ge& \frac{\|\nabla w^+(x_0)-\nabla w^+(y_0)\|}{\|x_0-y_0\|}\\
&\ge & -\frac{\|\nabla w^+(x_0)-\nabla w_P^+(x_0)\|}{\|x_0-y_0\|}-\frac{\|\nabla w^+(y_0)-\nabla w_P^+(y_0)\|}{\|x_0-y_0\|}\\
&+&\frac{\|\nabla w_P^+(x_0)-\nabla w_P^+(y_0)\|}{\|x_0-y_0\|}\\
&\ge& \frac{-2\eps}{\|x_0-y_0\|}+\kappa-\eps.
\eqq
Hence $\Lip(\nabla w,\partial V)\ge \kappa$.
\end{proof}

\section{Appendix}\label{chap 6}
\subsection{Recall the constructions of $w^+$.}\label{con:well+}
Let $A=\{p_1,...,p_m\}$ be a finite subset of $\mathbb{R}^n$, let $F\in\cF^1(A)$ and let $\kappa\ge\Gamma^1 (F;A)$.\\
For $p\in A$ we define:
\bqq
\widetilde p^+ &\triangleq& p-D_pf/\kappa,\hfill\\
d_p^+(x)&\triangleq& f_p-\frac{1}{2}(D_pf)^2/\kappa+\frac{1}{4}\kappa\|x-\widetilde p^+\|^2.
\eqq
When $S\subset A$ we define
\bqq
{d_S}^+(x) &\triangleq& \mathop {\inf }\limits_{p \in S} {d_p}^+(x),\\
\widetilde S^+ &\triangleq& \{\widetilde p^+:p \in S\},\\
\widehat  S^+&\triangleq&\text{ convex hull of } \widetilde S^+,\\
S_H^+&\triangleq&\text{smallest hyperplane containing } \widetilde S^+,\\
S_E^+&\triangleq&\{x:d_p^+(x)=d_{p'}^+(x) \text{ for all } p,p' \in S\},\\
S_{*}^+&\triangleq&\{x:d_p^+(x)=d_{p'}^+(x)\le d_{p''}^+(x) \text{ for all }p,p'\in  S,p''\in A\},\\
K^+&\triangleq&\{S:S\subset A \text{ and for some } x\in S_{*}^+,d_S^+(x)<d_{A-S}^+(x)\}.
\eqq
\begin{proposition}[\cite{Well1}, Lemma 3]\label{pr:cw+1} Let
$S^+_C=S^+_E\cap S^+_H$ for $S\in K^+$ then $S^+_C$ is a point. 
\end{proposition}
\begin{definition}\label{def:T+}
For all $S\in K^+$, set 
\bqq
T^+_S\triangleq\{x:x=\frac{1}{2}(y+z) \text {for some }y\in \widehat S^+ \text { and }z\in S^+_* \}
\eqq
\end{definition}
\begin{proposition}[\cite{Well1}, Lemma 15,17]\label{pr:cw+2} We have
$\bigcup\limits_{S \in K^+} {{T^+_S}}  = {\mathbb{R}^n}$ and $({{T^+_S}}\cap { {T^+_{S'}}})^0=\emptyset$ if $S\ne S'$.
\end{proposition}
\begin{definition}
$w^+_S(x)\triangleq d_S(S^+_C)+\frac{1}{2}\kappa d^2(x,S^+_H)-\frac{1}{2}\kappa d^2(x,S^+_E)$ for $S\in K^+$ and $x\in T^+_S$.
\end{definition}
\begin{definition}
$w^+(F,A,\kappa)(x)\triangleq w_S^+(x)$ if $x\in T^+_S$. 
\end{definition}
From \cite{Well1} we know that $w^+(F,A,\kappa)$ is well defined in $\cR^n$, $w^+(F,A,\kappa)\in \mathcal{C}^1(\cR^n,\R)$ and $\Lip(\nabla w^+(F,A,\kappa),\cR^n)\le \kappa$ .
\begin{theorem}\label{theo:W+} We have $w^+(F,A,\kappa)\in \mathcal{C}^{1,1}(\cR^n,\R)$ with $w^+(F,A,\kappa)(p)=f_p$, $\nabla w^+(F,A,\kappa)(p)=D_pf$ for all $p\in A$ and $\Lip(\nabla w^+(F,A,\kappa),\cR^n)\le \kappa$ .  

Further, if $g\in  \mathcal{C}^{1,1}(\cR^n,\R)$ with $g(p)=f_p$, $\nabla g(p)=D_pf$ when $p\in A$ and $\Lip(\nabla g,\cR^n)\le \kappa$, then $g(x)\le w^+(F,A,\kappa)(x)$ for all $x\in\R^n$.
\end{theorem}
\begin{proof}
Applying Proposition \ref{pro:halAB}, we have $\Gamma^1(F;\Omega)\le \kappa$ if and only if 
\bqq
\sqrt{A^2_{x,y}(F)+B^2_{x,y}(F)}+A_{x,y}(F)\le \kappa,\text{ }\forall x,y\in \Omega ,
\eqq
This inequality is equivalent to
\bqq
\frac{B_{x,y}^2(F)}{2M_1}+A_{x,y}(F)\le \frac{\kappa}{2},
\eqq
and hence it is equivalent to 
\bqq
f_y\le f_x+\frac{1}{2}\langle D_x f +D_y f,y-x\rangle +\frac{1}{4}\kappa\|y-x\|^2-\frac{1}{4\kappa}\|D_xf-D_yf\|^2,\text{ for any }x,y\in\Omega.
\eqq
Using \cite[Theorem 1]{Well1}, we finish the proof of this theorem.
\end{proof}

\begin{corollary}\label{cor:well+}
In the case $\kappa=\Gamma^1(F;A)$, let $W^+(F,A,\kappa)$ be the 1-field associated to $w^+(F,A,\kappa)$ then $W^+(F,A,\kappa)$ is an over extremal extension of $F$ on $\R^n$.
\end{corollary}
\begin{proof}
From  \cite[Proposition 2.4]{ELG1} we have $\Lip(\nabla w^+(F,A,\kappa),\cR^n)=\Gamma^1(W^+(F,A,\kappa),\cR^n)$. And so that the proof is immediate from Definition \ref{def:quantrongnhat} and Theorem \ref{theo:W+}. 
\end{proof}
\subsection{The constructions of $w^-$.}\label{con:well-}
By the same way, we can construct the function $w^-$ as follows.\\
Let $A=\{p_1,...,p_m\}$ be a finite subset of $\mathbb{R}^n$ and let $F\in\cF^1(A)$. Let $\kappa\ge\Gamma^1 (F;A)$.\\
For $p\in A$ we define:
\bqq
\widetilde p^- &\triangleq&p+D_pf/\kappa,\hfill\\
d_p^-(x)&\triangleq& f_p+\frac{1}{2}D_pf^2/\kappa-\frac{1}{4}\kappa\|x-\widetilde p^-\|^2.
\eqq
When $S\subset A$ we define
\bqq
{d_S^-}(x) &\triangleq& \mathop {\sup }\limits_{p \in S} {d_p^-}(x),\\
\widetilde S^- &\triangleq& \{\widetilde p^-:p \in S\},\\
\widehat  S^-&\triangleq&\text{ convex hull of } \widetilde S,\\
S_H^-&\triangleq&\text{smallest hyperplane containing } \widetilde S^-,\\
S_E^-&\triangleq&\{x:d_p^-(x)=d_{p'}^-(x) \text{ for all } p,p' \in S\},\\
S_{*}^-&\triangleq&\{x:d_p^-(x)=d_{p'}^-(x)\ge d_{p''}^-(x) \text{ for all }p,p'\in  S,p''\in A\},\\
K^-&\triangleq&\{S:S\subset A \text{ and for some } x\in S_{*}^-,d_S^-(x)>d_{A-S}^-(x)\}.
\eqq
\begin{proposition}\label{pr:cw-1} Set
$S^-_C\triangleq S^-_E\cap S^-_H$ for $S\in K^-$ then $S^-_C$ is a point. 
\end{proposition}
\begin{definition}\label{def:T-}
For all $S\in K^-$, set
\bqq
T^-_S\triangleq\{x:x=\frac{1}{2}(y+z) \text { for some }y\in \widehat S^- \text { and }z\in S^-_* \}
\eqq
\end{definition}
\begin{proposition}\label{pr:cw-2} We have
$\bigcup\limits_{S \in K^-} {{T^-_S}}  = {\mathbb{R}^n}$ and $({{T^-_S}}\cap { {T^-_{S'}}})^0=\emptyset$ if $S\ne S'$.
\end{proposition}
\begin{definition}
$w^-_S(x)\triangleq d_S(S^-_C)-\frac{1}{2}\kappa d^2(x,S^-_H)+\frac{1}{2}\kappa d^2(x,S^-_E)$ for $S\in K$ and $x\in T^-_S$.
\end{definition}
\begin{definition}
$w^-(F,A,\kappa)(x)\triangleq w_S^-(x)$ if $x\in T^-_S$. 
\end{definition}
\begin{theorem}\label{theo:W-} We have $w^-(F,A,\kappa)\in \mathcal{C}^{1,1}(\cR^n,\R)$ with $w^-(F,A,\kappa)(p)=f_p$, $\nabla w^-(F,A,\kappa)(p)=D_pf$ for all $p\in A$ and $\Lip(\nabla w^-(F,A,\kappa),\cR^n)\le \kappa$ .  

Further, if $g\in  \mathcal{C}^{1,1}(\cR^n,\R)$ with $g(p)=f_p$, $\nabla g(p)=D_pf$ when $p\in A$ and $\Lip(\nabla g,\cR^n)\le \kappa$, then $w^-(F,A,\kappa)(x)\le g(x)$ for all $x\in\R^n$.
\end{theorem}
\begin{corollary}
In the case $\kappa=\Gamma^1(F;A)$, let $W^-(F,A,\kappa)$ be the 1-field associated to $w^-(F,A,\kappa)$ then $W^-(F,A,\kappa)$ is an under extremal extension of $F$ on $\R^n$.
\end{corollary}
\subsection{Domain infinite.}\label{sec:casedoinf}
Let $\Omega$ be a nonempty subset of $\cR^n$ and $F\in\cF^1(\Omega)$. Fix $\kappa=\Gamma^1(F;\Omega)$.\\
We call $\mathcal{P}$ to be the set of all finite subsets of $\Omega$.  Applying Theorem \ref{theo:W+} anh Theorem \ref{theo:W-}, for any $x\in \cR^n$, and for any $P,P'\in\mathcal{P}$ satisfying $P\subset P'$ we have  
$$w^-(F,P,\kappa)(x)\le w^-(F,P',\kappa)(x)\le w^+(F,P',\kappa)(x)\le w^+(F,P,\kappa)(x).$$
So that we can define $$w^+(F,\Omega)(x)=\inf\limits_{P\in \mathcal{P}}w^+(F,P,\kappa)(x),$$ and $$w^-(F,\Omega)(x)=\sup\limits_{P\in \mathcal{P}}w^-(F,P,\kappa)(x).$$
\begin{theorem}\label{theo:tqwells}Let $W^+(F,\Omega)$ be the 1-field associated to $w^+(F,\Omega)$ and let $W^-(F,\Omega)$ be the 1-field associated to $w^-(F,\Omega)$. Then $W^+(F,\Omega)$ is an over extremal extension of $F$ on $\R^n$ and $W^-(F,\Omega)$ is an under extremal extension of $F$ on $\R^n$.
\end{theorem}
\begin{proof}
Using [\cite{Well1}, Theorem 2], the proof is similar as Theorem \ref{theo:W+} and Corollary \ref{cor:well+}. 
\end{proof}

\section{Acknowledgements}

E.L.G. and  T.V.P. are partially supported by the ANR (Agence Nationale de la
Recherche) through HJnet projet ANR-12-BS01-0008-01.

\end{document}